\documentclass[11pt]{article}
\usepackage{amsmath,amsfonts,amssymb,amsthm}
\usepackage{mathtools,extarrows} 
\usepackage{xcolor}
\usepackage{hyperref}
\usepackage{dsfont}

\usepackage{booktabs}

\usepackage{colortbl}
\newcommand{\myrowcolour}{\rowcolor[gray]{0.925}}

\allowdisplaybreaks

\newtheorem{proposition}{Proposition}[section]
\newtheorem{theorem}[proposition]{Theorem}
\newtheorem{corollary}[proposition]{Corollary}
\newtheorem{lemma}[proposition]{Lemma}
\newtheorem{remark}[proposition]{Remark}

\newtheorem{example}[proposition]{Example}

\newcommand{\nc}{\newcommand}
\nc{\I}{{\mathbf 1}}

\nc{\bN}{{\mathbf N}}
\nc{\bM}{{\mathbf M}}
\nc{\cB}{{\mathcal B}}
\nc{\cM}{{\mathcal M}}
\nc{\R}{{\mathbb R}}
\nc{\N}{{\mathbb N}}
\nc{\Z}{{\mathbb Z}}
\nc{\BX}{{\mathbb X}}
\nc{\BY}{{\mathbb Y}}
\nc{\cX}{{\mathcal X}}
\nc{\cY}{{\mathcal Y}}
\nc{\cN}{{\mathcal N}}
\nc{\cF}{{\mathcal F}}

\DeclareMathOperator{\supp}{supp}

\DeclareMathOperator*{\esssup}{\mathrm{ess\, sup}}

\nc{\tv}{d_{\mathrm TV}}

\nc{\BP}{\mathbb{P}}
\nc{\BE}{\mathbb{E}}
\nc{\BQ}{\mathbb{Q}}

\numberwithin{equation}{section}

\begin{document} 

\renewcommand{\thefootnote}{\fnsymbol{footnote}}
\author{{\sc {\sc Steffen Betsch\footnotemark[1]} \hspace{0.1mm} and G\"unter Last\footnotemark[2]} \\ {\small Karlsruhe Institute of Technology, Institute of Stochastics, 76131 Karlsruhe, Germany.}}
\footnotetext[1]{steffen.betsch@kit.edu
}
\footnotetext[2]{guenter.last@kit.edu
}

\title{On the uniqueness of Gibbs distributions with a non-negative 
and subcritical pair potential} 
\date{\today}
\maketitle

\begin{abstract} 
\noindent 
We prove that the distribution of
a Gibbs process with non-negative pair potential
is uniquely determined as soon as an associated Poisson-driven
random connection model (RCM) does not percolate.
Our proof combines disagreement coupling in continuum
(established in \cite{HTHou17,LastOtto21}) with a coupling 
of a Gibbs process and a RCM. The improvement over previous uniqueness results is illustrated both in theory and simulations.

\end{abstract}


\vspace{0.1cm}

\section{Introduction}\label{sintro}

Let $(\mathbb{X}, d)$ be a complete separable metric space, denote its Borel-$\sigma$-field by $\mathcal{X}$, and let $\lambda$ be
a locally finite measure on $\BX$.
Denote by $v\colon \mathbb{X}\times\mathbb{X}\to [0,\infty]$ 
a (non-negative) {\em pair potential}, that is, a measurable
and symmetric function. 
A {\em Gibbs process} with pair potential $v$ and
{\em reference measure} $\lambda$ is a {\em point process} $\eta$ on
$\BX$ satisfying, loosely speaking,  
\begin{align*}
\BP\big(\eta(\mathrm{d}x)>0\mid \eta_{\BX\setminus\{x\}}\big)
=\exp\bigg(-\int_{\BX} v(x,y)\,\mathrm{d}\eta(y)\bigg) \mathrm{d}\lambda(x),
\end{align*}
for all $x\in\mathbb{X}$.  Here we interpret
a point process $\eta$ as a random measure on $\mathbb{X}$,
which is integer-valued on bounded sets,
and we write $\eta_B$ for the restriction of $\eta$ to a set $B\in\cX$.
For more details on the point process approach to Gibbs
processes we refer to \cite{NgZe79,MaWaMe79,Dereudre18}
and to Subection \ref{subGibbs}. 
We let $\mathcal{G}(v,\lambda)$ denote the space of distributions
of Gibbs processes with pair potential $v$ and reference measure $\lambda$.

Our assumption below will imply that
\begin{align}\label{e1.3}
\int_{\mathbb{X}} \big(1-e^{-v(x,y)}\big)\,\mathrm{d}\lambda(y) < \infty,\quad \lambda\text{-a.e.\ $x\in\BX$}.
\end{align}
It is then known that $\mathcal{G}(v,\lambda)\neq\varnothing$,
referring to \cite{Ruelle70} for the case $\mathbb{X}=\R^d$ and to \cite{Jansen19}
for the present generality. In this paper we establish an
explicit criterion for the uniqueness of a Gibbs distribution. To do
so we consider the {\em random connection model} (RCM) (see e.g.\
\cite{Pen91,MeesterRoy}) with {\em connection function}
$\varphi:=1-\exp(-v)$ based on a Poisson process $\Phi$ with intensity
measure $\lambda$. In a RCM every pair $x,y$ of distinct points from
$\Phi$ (taking into account multiplicities) is connected with probability $\varphi(x, y)$, independently
for different pairs. This gives a random graph $\Gamma$ whose vertices
are the points of $\Phi$.  For each $x\in\BX$ we consider a RCM
$\Gamma^x$ with connection function $\varphi$ based on the Poisson
process $\Phi$ augmented by the point $x$, as detailed in Subsection
\ref{subrcm}.  Let $C_x$ denote the cluster (component) of this graph
containing $x$, interpreted as a point process on $\BX$.
We say that the RCM is {\em subcritical} if these clusters 
have only a finite number of points, that is,
\begin{align}\label{esubcritical}
\BP\big(C_x(\BX)<\infty\big)=1,\quad \lambda\text{-a.e.\ $x\in\BX$}.
\end{align}
In this case we also say that the pair $(v,\lambda)$ is subcritical.
Note that the degree of $x$ in the RCM $\Gamma^x$ has a Poisson distribution
with parameter $\int_{\mathbb{X}} \big(1-e^{-v(x,y)}\big)\,\mathrm{d}\lambda(y)$.
Therefore \eqref{esubcritical} implies the integrability condition \eqref{e1.3}.

The following theorem is the main result of our paper.
It gives an affirmative answer to a question asked
in \cite{Jansen19}.
By $|A|$ we denote the cardinality of a set $A$.

\begin{theorem}\label{t2} Assume that $(v,\lambda)$ is subcritical.
Then $|\mathcal{G}(v,\lambda)|=1$.
\end{theorem}

From a percolation perspective it might be helpful to
consider for each $\gamma\ge 0$ a RCM with connection function
$\varphi$ based on a Poisson process with intensity measure $\gamma\lambda$.
Then $\gamma$ can be interpreted as an {\em intensity} 
(or {\em activity}). Define a {\em critical} intensity by  
\begin{align}\label{criticalactivity}
\gamma_c:=\sup\{\gamma\ge 0:\text{$(v,\gamma\lambda)$ is subcritical}\}.
\end{align}
Theorem \ref{t2} immediately implies the following result.

\begin{corollary}\label{c11} Assume that $\gamma<\gamma_c$.
Then $|\mathcal{G}(v,\gamma\lambda)|=1$.
\end{corollary}

If $\mathbb{X}=\R^d$, $\lambda$ is the Lebesgue measure, and $v$ is translation invariant,
then it is known 
that $0<\gamma_c<\infty$, see \cite{MeesterRoy}. But even in this case
Theorem \ref{t2} raises the question for sufficient conditions
implying that $(v,\lambda)$  is subcritical.
By comparison with a branching process (similar to \cite{Pen91})
we prove in Appendix \ref{sbranching} that
\begin{align*}
\esssup_{x\in\BX}\int_{\mathbb{X}} \varphi(x, y) \,\mathrm{d}\lambda(y)<1
\end{align*}
is sufficient for subcriticality (where the essential supremum is with respect to $\lambda$).
Hence we obtain the following corollary, still on the general state space.

\begin{corollary}\label{c1} Assume that
\begin{align}\label{ebrlow2}
\gamma\cdot\esssup_{x\in\BX}\int_{\mathbb{X}} \varphi(x, y) \,\mathrm{d}\lambda(y)<1.
\end{align}
Then $|\mathcal{G}(v,\gamma\lambda)|=1$.
\end{corollary}

Corollary \ref{c1} is the main result in the recent preprint
\cite{HouZass21}, proved there under two additional assumptions
by applying the classical Dobrushin method \cite{Dobrushin68}
to a suitable discretization.
Our Theorem \ref{t2} generalizes this in several way. First of all
we work on a general space $\BX$ and not just on $\R^d$.
Second we do not need any {\em hard core} assumption on $v$.
Neither do we need another technical assumption made in
\cite{HouZass21}. Finally, and perhaps most importantly, the simulations in Section \ref{simulations}
show that we extend the bounds on the uniqueness region much beyond the branching bound
\eqref{ebrlow2}.

Theorem \ref{t2} also considerably generalizes the percolation
criteria derived in \cite{HTHou17} for translation invariant pair potentials in $\R^d$ with
a {\em finite range} $R>0$, that is, $v(x)=0$ for $|x|>R$.
Indeed, the uniqueness result in \cite[Subsection 3.2.1]{HTHou17} 
requires the RCM with connection function 
$\varphi_R(x):=\I\{ |x|\le R\}$
to be subcritical. Since $\varphi\leq\varphi_R$, the critical
intensity associated with $\varphi$ is larger than
the one associated with $\varphi_R$.
As our bound takes into
account the full information contained in the pair potential 
and not just its range, the difference can be very significant as illustrated in Section \ref{simulations}.
The RCM associated with
$\varphi_R$ is referred to as the {\em Gilbert graph}, see e.g.\ \cite{LastPenrose17}.
The corresponding pair potential is $v_R(x)=\infty \cdot\I\{ |x|\leq R\}$,
and describes hard spheres of radius $R/2$ (in equilibrium), where
$\infty \cdot 0=0\cdot\infty := 0$. 
A non-Poisson version of disagreement percolation
was applied in \cite{DeHou21} to prove
uniqueness of the Widom-Rowlinson process
for a certain range of parameters. This Gibbs process is not governed
by a pair potential but enjoys nice and rather specific
monotonicity properties.

We think of Corollary \ref{c11} 
as of an explicit lower bound on the range of uniqueness.
Even though $\gamma_c$ is not explicitly known,
it admits a clear probabilistic meaning and
the Poisson-based RCM can be easily simulated.
To get an impression of the critical
intensities, we have simulated three examples, including
the well-studied Gilbert graph as a benchmark model as well as two interaction functions from the physics literature.
The results are presented in Section \ref{simulations}.

The Dobrushin criterion (mentioned above) was also used in other
papers to establish uniqueness of Gibbs measures, for
instance in \cite{GeoHaeg96} (without giving the details) and
in \cite{CoDaKoPa18}. The results of these papers allow the pair potential
to take negative values but do not provide explicit information
on the domain of uniqueness. A further drawback is the assumption
of a finite range, a severe restriction of generality.
Another method for proving uniqueness of Gibbs distributions are fixed point
methods based on  Kirkwood–Salsburg integral equations
and the fact that the correlation functions determine the distribution of a point
process under suitable assumptions. This method can be traced back to \cite{Ruelle69}.
For some recent contributions we refer to \cite{Jansen19}, which might also serve as a good survey, and \cite{Zass21}.
The uniqueness intensity region  identified 
by this method is characterized by the contractivity
of certain integral operators and does not seem to have an explicit probabilistic
interpretation. In view of the method and the special cases discussed
in \cite{Jansen19}, we expect this region to be comparable with 
the branching bound \eqref{ebrlow2} or its generalization in Appendix \ref{sbranching}.
Still another method for proving uniqueness is to
identify a Gibbs distribution as a stationary and reversible measure
with respect to a suitable Markovian dynamics, cf. \cite{FFG02,SY13}.
Uniqueness follows if the so-called ancestor clans, coming from an embedding into a space-time
Poisson process, are finite. The resulting bounds on the domain of
uniqueness are not explicit but might be comparable with the branching bounds
(see also the discussion in \cite{BHLV20}).

Our main tool for proving Theorem \ref{t2} 
is the disagreement coupling from \cite{HTHou17,LastOtto21}, which was inspired by the results of \cite{BergMaes94} in the discrete setting,
combined with an (approximative) simultaneous coupling with the associated RCM. 
The connection function $\varphi=1-e^{-v}$ was used in \cite{GivenStell90,GeoHaeg96} for a 
Fortuin-Kasteleyn type  coupling 
of a {\em Potts model} (allowing for general pairwise interactions between particles
of different types) and the 
{\em continuum cluster model}. While this coupling
proceeds by conditioning, we use a marked Poisson process
for embedding two Gibbs processes with different boundary
conditions and, at the same time, for constructing a RCM.
This way we can directly refer to the percolation properties
of a Poisson-driven RCM. 
We need to assume that $v\ge 0$. Apart from that,
and the subcriticality, we do not make any further assumptions
such as finite range or strict repulsion.
We believe that our approach will be useful
for analyzing further properties of repulsive Gibbs processes,
even though its details are a bit technical.

\section{Preliminaries}\label{preliminaries}

In this section we collect a few basic definitions and facts
on Gibbs point processes and the random connection model.
A reader familiar with these topics might skip this section
at first reading.

\subsection{Gibbs processes}\label{subGibbs}

Let $(\BX,\mathcal{X})$ denote a Borel space and $\lambda$ a
$\sigma$-finite measure on $\BX$. Let $B_1 \subset B_2 \subset \cdots$
be measurable sets of finite $\lambda$-measure such that
$\bigcup_{\ell = 1}^\infty B_\ell = \BX$.  Let $\cX_b$ denote the
system of all bounded Borel sets, meaning the collection of all sets
which are contained in one of the $B_\ell$.  A measure $\nu$ on
$(\BX, \cX)$ which is finite on $\cX_b$ is called {\em locally
  finite}.  Let $\bN(\BX)$ be the space of all locally finite measures
on $\BX$ which are $\N_0$-valued on $\mathcal{X}_b$, and let
$\cN(\BX)$ denote the smallest $\sigma$-field such that
$\mu\mapsto \mu(B)$ is measurable for all $B\in\cX$. A {\em point
  process} on $\mathbb{X}$ is a random element $\eta$ of
$\bN(\mathbb{X})$, defined over some fixed probability space
$(\Omega,\cF,\BP)$.  The {\em intensity measure} of $\eta$ is the
measure $\BE[\eta]$ defined by $\BE[\eta](B):=\BE[\eta(B)]$,
$B\in\cX$.  By our assumption on $\BX$, every point process $\eta$ is
{\em proper}, that is,
$ \eta = \sum_{n=1}^{\eta(\mathbb{X})}\delta_{X_n}$, where
$\{X_n : n\geq 1\}$ is a collection of random elements with values in
$\mathbb{X}$, see \cite[Section 6.1]{LastPenrose17} for more
details. Given a measure $\nu$ on $\BX$ and $B\in \mathcal X$, we
write $\nu_B$ to indicate the trace measure $ D\mapsto \nu( D\cap B)$.
For $\mu \in \bN(\mathbb{X})$ we write $x \in \mu$ if $\mu(\{x\})>0$.

Let $\kappa\colon\BX\times\bN(\mathbb{X})\to[0, \infty)$ be measurable
and satisfy the cocycle relation,
\begin{align*}
	\kappa(x, \mu) \cdot \kappa(y, \mu + \delta_x)
	= \kappa(y, \mu) \cdot \kappa(x, \mu + \delta_y), 
\quad x, y \in \mathbb{X}, \, \mu \in \mathbf{N}(\mathbb{X}) .
\end{align*}
A point process $\eta$ on $\BX$ is called a {\em Gibbs process}
with {\em Papangelou intensity} (PI) $\kappa$ (and {\em reference measure} $\lambda$) if
\begin{align}\label{eGNZ}
\BE\bigg[\int_{\mathbb{X}} f(x,\eta)\,\mathrm{d}\eta(x)\bigg]=
\BE\bigg[ \int_{\mathbb{X}} f(x,\eta+\delta_x) \, \kappa(x,\eta)\,\mathrm{d}\lambda(x)\bigg]
\end{align}
for each measurable $f\colon\BX\times\bN(\mathbb{X})\to[0, \infty)$.
The latter are the {\em GNZ equations} named after Georgii, Nguyen and Zessin
\cite{Georgii76,NgZe79}.
If $\kappa\equiv 1$ this is the {\em Mecke equation}, characterizing a
{\em Poisson process} with intensity measure $\lambda$, cf. 
\cite[Theorem 4.1]{LastPenrose17}. In that case the distribution of
$\eta$ is denoted by $\Pi_\lambda$.

Equation \eqref{eGNZ} can be generalized. For $m\in\N$ we define
a measurable map $\kappa_m\colon \BX^m\times\bN(\mathbb{X})\to [0, \infty)$ by
\begin{align*}
 \kappa_m(x_1,\ldots,x_m,\mu)
:=\kappa(x_1,\mu)\cdot\kappa(x_2,\mu+\delta_{x_1})\cdots\kappa(x_m,\mu+\delta_{x_1}+\cdots+\delta_{x_{m-1}}).
\end{align*}
Note that $\kappa_1=\kappa$.
If $\eta$ is a Gibbs process with PI $\kappa$ then
we have for each $m\in\N$ and  for each measurable $f\colon\BX^m\times\bN(\mathbb{X})\to[0, \infty)$ that
\begin{align}\label{eGNZmulti}
\BE\bigg[\int_{\mathbb{X}^m} &f(x_1,\ldots,x_m,\eta)\, \mathrm{d}\eta^{(m)}(x_1,\ldots,x_m)\bigg] \notag \\ 
&=\BE \bigg[\int_{\mathbb{X}^m} f(x_1,\ldots,x_m,\eta+\delta_{x_1}+\cdots+\delta_{x_m}) \cdot \kappa_m(x_1,\ldots,x_m,\eta)
\,\mathrm{d}\lambda^m(x_1,\ldots,x_m)\bigg],
\end{align}
where $\mu^{(m)}$ is the $m$-th factorial measure of $\mu$, consulting
\cite{LastPenrose17} for the general definition.  This follows by
induction, using that
\begin{align*}
\kappa_{m+1}(x_1,\ldots,x_{m+1},\mu)=\kappa_m(x_1,\ldots,x_m,\mu) \cdot
\kappa(x_{m+1},\mu+\delta_{x_1}+\cdots+\delta_{x_m}). 
\end{align*}

By the cocycle assumption on $\kappa$, $\kappa_m$ is symmetric in the first $m$ arguments.
The {\em Hamiltonian}
$H\colon\bN(\mathbb{X})\times\bN(\mathbb{X})\to (-\infty,\infty]$ 
(based on $\kappa$) is defined by
\begin{align*}
 H(\mu,\nu):=
\begin{cases}
0,&\text{if $\mu(\BX)=0$}, \\
 -\log \kappa_m(x_1,\ldots,x_m,\nu),&
 \text{if $\mu=\delta_{x_1}+\cdots+\delta_{x_m}$}, \\
 \infty,&\text{if $\mu(\BX)=\infty$}.
 \end{cases}
\end{align*}
For $B\in\cX_b$ the {\em partition function}
$Z_B\colon\bN(\mathbb{X})\to[0,\infty]$ is defined by
\begin{align}\label{epartition}
  Z_B(\nu):= e^{\lambda(B)} \int_{\mathbf{N}(\mathbb{X})} 
e^{-H(\mu,\nu)}\,\mathrm{d}\Pi_{\lambda_B}(\mu)
             = 1 + \sum_{m = 1}^\infty \frac{1}{m!} \int_{B^m} \kappa_m(x_1, \dots, x_m, \nu) 
\, \mathrm{d}\lambda^m(x_1, \dots, x_m),
\end{align}
where \cite[Exercise 3.7]{LastPenrose17} establishes the equality.
We clearly have that $Z_B(\nu)\ge 1$.
For $\nu\in\bN(\mathbb{X})$ and $B \in \mathcal{X}_b$,
the {\em Gibbs measure} $\mathrm{P}_{B, \nu}$ on $\bN(\mathbb{X})$ is defined by
\begin{align*}
  \mathrm{P}_{B, \nu} := Z_B(\nu)^{-1} \, e^{\lambda(B)}\int_{\bN(\mathbb{X})} 
\I\{\mu\in\cdot\} \, e^{-H(\mu,\nu)}\, \mathrm{d}\Pi_{\lambda_B}(\mu),
\end{align*}
provided that $Z_B(\nu)<\infty$, and an expansion similar to
\eqref{epartition} is possible. If $Z_B(\nu)=\infty$ we set
$\mathrm{P}_{B, \nu} \equiv 0$.  The measure $\mathrm{P}_{B, \nu}$
is concentrated on $\bN_B(\mathbb{X})$, the set of all measures
$\mu \in \mathbf{N}(\mathbb{X})$ with $\mu(B^c) = 0$, where
$B^c := \mathbb{X} \setminus B$, and $\mathrm{P}_{B, \nu}$ is the distribution of a Gibbs process with PI
$\kappa^{(B, \nu)}(x, \mu) := \kappa(x, \nu + \mu) \, \I_B(x)$ (and
reference measure $\lambda$).  It was proved in \cite{MaWaMe79,NgZe79}
that if $\eta$ is a Gibbs process with PI $\kappa$ then
\begin{align*}
\BP\big(Z_B(\eta_{B^c})<\infty\big)=1,\quad B\in\cX_b,
\end{align*}
and, for each measurable $f\colon\bN(\mathbb{X})\to [0, \infty)$,
\begin{align}\label{edlr}
 \BE\big[f(\eta_B)\mid \eta_{B^c}\big]
  =\int_{\mathbf{N}(\mathbb{X})} f(\mu) \, \mathrm{d}\mathrm{P}_{B, \eta_{B^c}}(\mu),
 \quad B\in\cX_b,
\end{align}
where relations involving conditional expectations are assumed
to hold almost surely.
These are the {\em DLR-equations}, cf. \cite{Ruelle70,Kallenberg17,Mase00}.

The Gibbs distributions we deal with are based on a non-negative pair potential
$v\colon\BX \times \BX\to [0,\infty]$, a symmetric and measurable function.
In this case we define the corresponding PI $\kappa$ by
\begin{align*}
\kappa(x,\mu):=\exp\bigg(- \int_{\mathbb{X}} v(x,y)\, \mathrm{d}\mu(y)\bigg),\quad x \in\BX, \, \mu \in\bN(\BX).
\end{align*}
If \eqref{e1.3} holds, then a Gibbs process with this
PI is known to exist, referring to \cite{Ruelle70} for the case $\BX=\R^d$,
and to \cite{Jansen19}  
for the more general case considered here.
If $v$ is allowed to take negative values the existence proofs
become more complicated, see e.g. \cite{Ruelle70,Mase00,DeDrGeorgii12,DVass19}.

\subsection{The random connection model}\label{subrcm}

In the setting of Section \ref{subGibbs}, suppose that
$\varphi\colon\BX \times \BX\to[0,1]$ is a measurable and symmetric
function. Let $\Phi=\sum^{\Phi(\BX)}_{n=1}\delta_{X_n}$ be a point
process.  Let $U_{i,j}$, $i,j\in\N$, be independent random variables,
uniformly distributed on the unit interval $[0,1]$, and such that the
double sequence $(U_{i,j})_{i, j \in \N}$ is independent of $\Phi$.  Let $\prec$ be
an order on $\BX$ with
$\{ (x, y) \in \BX^2 : x \prec y \} \in \mathcal{X}^{\otimes 2}$.  Let
$\BX^{[2]}$ denote the space of all sets $e\subset\BX$ containing
exactly two elements, which is a measurable space in its own right. We
define a point process $\Gamma$ on $\BX^{[2]}$ by
\begin{align} \label{eGamma}
\Gamma:=\sum^{\Phi(\BX)}_{i,j=1}
\I\big\{X_i\prec X_j, \,U_{i,j}\le \varphi(X_i,X_j)\big\} \cdot
\delta_{\{X_i,X_j\}}.
\end{align}
This is the {\em random connection model} (RCM) (based on $\Phi$) with 
{\em connection function} $\varphi$. 
We interpret $\Gamma$ as a random graph with vertex set $\Phi$.
As a rule, there are isolated points from $\Phi$ with no emanating edges.
While the definition of $\Gamma$ depends on the ordering
of the points of $\Phi$, its distribution does not.

We say that $x,y\in\Phi$ are connected
via $\Gamma$ if either $x=y$ or there exist $n\in\N$ and
$e_1,\ldots,e_n\in\Gamma$ such that $x\in e_1$, $y\in e_n$ and
$e_i\cap e_{i + 1}\ne\varnothing$ for $i \in \{ 1, \dots, n-1 \}$.  In
this case we write $x \xleftrightarrow{\Gamma} y$.  
Let $\Gamma^z$ be a random connection model based on
$\Phi^z:=\Phi+\delta_z$.
The cluster of $z$
in $\Gamma^z$ is the point process $C_z$ on $\BX$ given by
\begin{align*}
  C_z:=\int_{\mathbb{X}} \I\{x\in\cdot \} \, \I\{z \xleftrightarrow{\Gamma^z} x\} \, \mathrm{d}\Phi^z(x).
\end{align*}
It charges $z$ and all points from $\Phi$ which are connected to $z$ via $\Gamma^z$. 

Let us now assume that $\Phi$ is a Poisson
process with intensity measure $\lambda$.
Then we say that the RCM is {\em subcritical} if
\eqref{esubcritical} holds.
In this case we also say that the pair $(\varphi,\lambda)$ is subcritical.
If $\varphi=1-e^{-v}$ for a pair potential $v$, then we
say that the pair $(v,\lambda)$ is subcritical.


\begin{remark}\rm
	Assume that $\Phi$ is a Poisson process with a non-diffuse intensity measure
	$\lambda$. Then some of the points of $\Phi$ come
	with multiplicities, and there are other, more refined, ways to define a RCM.
	Indeed when connecting points (taking into account multiplicities)
	we may introduce multiple edges between their positions. And we may also allow for
	(multiple) loops. The result would be a random multigraph.
	To avoid such technicalities we stick to the more simple and intuitive definition
	\eqref{eGamma} and the notion of connectedness introduced thereafter.
\end{remark}

\section{Proof of Theorem \ref{t2}}\label{sproof}

The existence part of Theorem \ref{t2} is settled by Theorem B.1 of \cite{Jansen19}.
Therefore, it suffices to show that $|\mathcal{G}(v,\lambda)|\le 1$.
We first argue that we can assume, without loss of generality, that $\lambda$
is diffuse. To do so, we use {\em randomization}, a well-known technique in
point process theory, which was already tailored to Gibbs processes in \cite{SchuhStuck14}. We consider the space $\tilde\BX:=\BX\times[0,1]$ and equip it with the product metric and the product $\tilde\lambda$ 
of $\lambda$ and the Lebesgue measure on $[0,1]$.
Define $\tilde v\colon\tilde\BX\times\tilde\BX\to [0,\infty]$ by
$\tilde v\big( (x,r),(y,s) \big):=v(x,y)$.
Let $\tilde C_{(z, r)}$ denote the cluster of $(z, r) \in \tilde \BX$ in a RCM based on $\tilde \Phi^{(z, r)}$, where $\tilde \Phi :=\sum_{n=1}^{\Phi(\mathbb{X})}\delta_{(X_n,U_n)}$ is a {\em uniform randomization} of the Poisson process $\Phi = \sum_{n=1}^{\Phi(\mathbb{X})}\delta_{X_n}$ on $\BX$ with intensity measure $\lambda$. Here the randomization is defined with the help of independent random variables $U_n$, $n \in \N$, that are uniformly distributed on $[0,1]$, with the whole sequence $(U_n)_{n \in \N}$ being independent of $\Phi$. The process $\tilde \Phi$ is a Poisson process on $\tilde \BX$ with intensity measure $\tilde\lambda$, and we have
\begin{align*}
	\tilde C_{(z, r)}(\BX)
	\leq C_z(\BX)  \quad \mathbb{P}\text{-a.s.} ,
\end{align*}
with $C_z$ defined in Section \ref{subrcm}.
Hence, if $(v, \lambda)$ is subcritical then so is $(\tilde v, \tilde \lambda)$.
Moreover, if $\eta$ is a Gibbs process in $\BX$ with PI $\kappa$ and reference measure $\lambda$, it is easy to show via the GNZ-equations \eqref{eGNZ} that these equations hold with $(\eta,\lambda,\kappa)$ replaced by 
$(\tilde\eta,\tilde\lambda,\tilde\kappa)$, where $\tilde \eta$ is a uniform randomization of $\eta$ and $\tilde\kappa\big( (x,r), \psi \big) = \kappa\big( x, \psi(\cdot\times[0,1]) \big)$. In particular, each uniform randomization of $\eta$ is a Gibbs process
with PI $\tilde\kappa$ and reference measure $\tilde\lambda$. Thus, if $\eta$ is a point process with $\mathbb{P}^\eta \in \mathcal{G}(v,\lambda)$, where $\mathbb{P}^X$ denotes the distribution of a random element $X$, then $\mathbb{P}^{\tilde\eta} \in \mathcal{G}(\tilde v,\tilde\lambda)$ for any uniform randomization $\tilde \eta$ of $\eta$. Now, if $(v, \lambda)$ is subcritical and Theorem \ref{t2} holds for diffuse reference measures, then $|\mathcal{G}(\tilde v,\tilde\lambda)| = 1$ as $(\tilde v, \tilde \lambda)$ is subcritical. Consequently, uniform randomizations $\tilde \eta, \tilde \eta'$ of two point processes $\eta, \eta'$ with $\mathbb{P}^\eta, \mathbb{P}^{\eta'} \in \mathcal{G}(v,\lambda)$ satisfy $\mathbb{P}^{\tilde\eta} = \mathbb{P}^{\tilde \eta'}$, and we obtain, for each $A \in \mathcal{N}(\BX)$,
\begin{align*}
	\mathbb{P}(\eta \in A)
	= \mathbb{P}\big( \tilde\eta(\cdot \times [0, 1]) \in A \big)
	= \mathbb{P}\big( \tilde\eta'(\cdot \times [0, 1]) \in A \big)
	= \mathbb{P}(\eta' \in A).
\end{align*}
We conclude that $\mathbb{P}^\eta = \mathbb{P}^{\eta'}$, that is, $|\mathcal{G}(v,\lambda)| = 1$.

In the remainder of the section we 
let the space $(\BX, d)$ and the pair $(v,\lambda)$ be as in the introduction and assume, in addition,
that $\lambda$ is diffuse. The definition of $\mathbf{N}(\BX)$ (see Section \ref{subGibbs}) is based on the collection $\mathcal{X}_b$ of $d$-bounded Borel subsets of $\BX$. For the moment we dispense with the subcriticality assumption on $(v, \lambda)$ and suppose that only the weaker assumption \eqref{e1.3} holds.
It is an easy exercise to construct sets
$B_1 \subset B_2 \subset \cdots$ in $\mathcal{X}_b$ such that
$\bigcup_{\ell = 1}^\infty B_\ell = \BX$ and
\begin{align} \label{eq. 3.1}
	\int_{B_\ell} \int_{\BX} \big( 1 - e^{-v(x, y)} \big) 
\mathrm{d}\lambda(y) \, \mathrm{d}\lambda(x)< \infty , \quad \ell \in \N.
\end{align}
We denote by $\mathcal{X}_b^*$ the
collection of all sets from $\mathcal{X}$ which are contained in one
of the $B_\ell$. The collection $\mathcal{X}_b^*$ is a ring over $\BX$
with $\sigma(\mathcal{X}_b^*) = \mathcal{X}$. Moreover, the algebra
\begin{align*}
	\mathcal{Z} := \bigcup_{C \in \mathcal{X}_b^*} \mathcal{N}_C(\BX)
\end{align*}
generates $\mathcal{N}(\BX)$, where we denote by $\mathcal{N}_C(\BX)$
the sub-$\sigma$-field of $\mathcal{N}(\BX)$ generated by all maps
$\mu \mapsto \mu(D)$ for $D \in \mathcal{X}$ with $D \subset C$. This precise construction of $\mathcal{X}_b^*$ is motivated by
\cite[Equation (B.6)]{Jansen19} and it proves useful later on in order
to avoid additional integrability assumptions.

Put $B_0 := \varnothing$. For each $\ell \in \N$ we consider a Borel
isomorphism $\iota_\ell$ from the Borel space
$\big( B_\ell \setminus B_{\ell - 1}, \mathcal{X} \cap (B_\ell \setminus B_{\ell - 1}) \big)$
onto a Borel subset of $(\ell - 1,\ell]$. Define the injective and measurable map
$\iota : \mathbb{X} \to (0, \infty)$ as
\begin{align*}
	\iota(x) := \sum_{k = 1}^\infty \iota_k(x) \, \I_{B_k \setminus B_{k - 1}}(x)
\end{align*}
and define an order on $\mathbb{X}$ via $x \prec y$ iff
$\iota(x) < \iota(y)$. Observe that for $x \in B_\ell$ and
$y \in B_\ell^c$ we always have $x \prec y$.

Put $\varphi := 1 - e^{-v}$ as a function on
$\mathbb{X} \times \mathbb{X}$. In order to apply the disagreement coupling from \cite{LastOtto21}, we want to approximate the RCM with connection function $\varphi$ by suitably interpreting a Poisson process on a rich product space. To this end, we consider as a mark space
$\mathbb{M} := [0, 1]^{\mathbb{N} \times \mathbb{N}}$, the space of
doubly indexed sequences in $[0, 1]$, endowed with the product
$\sigma$-field, and denote by $\mathbb{Q}$ the probability measure on
$\mathbb{M}$ given as an infinite product of uniform distributions on
$[0, 1]$. Moreover, we need to be able to separate the points in our space with the help of countable partitions. Thus, for each $\delta > 0$, we let
$D_1^\delta, D_2^\delta, \dotso \in \mathcal{X}$ be a partition of
$\mathbb{X}$ such that for any two points $x, y \in \BX$ there exists some $\delta_0 > 0$ so that $x$ and $y$ are separated by the $\delta$-partition for every $\delta < \delta_0$, where the points being separated means that they lie in different sets of the partition. Such directed partitions can always be constructed in separable metric spaces.
Define the measurable map
$R_\delta\colon (\mathbb{X} \times \mathbb{M})^2 \to [0, 1]$ as
\begin{align*}
	R_\delta(x, r, y, s)
	:= \sum_{i, j = 1}^\infty \I_{D_i^\delta}(x) 
\I_{D_j^\delta}(y) \, \big( \I\{x \prec y\} \cdot r_{i,j} + \I\{y \prec x\} \cdot s_{j,i} \big)
\end{align*}
as well as a relation $\sim_{\delta}$ on $\mathbb{X} \times \mathbb{M}$ via
\begin{align*}
	(x, r) \sim_{\delta} (y, s) \quad \Longleftrightarrow \quad R_\delta(x, r, y, s) \leq \varphi(x, y) .
\end{align*}

Similar to $\mathbf{N}(\mathbb{X})$, denote by
$\mathbf{N}(\mathbb{X} \times \mathbb{M})$ the set of all measures
$\psi$ on $\mathbb{X} \times \mathbb{M}$ such that
$\psi(B \times \mathbb{M}) \in \mathbb{N}_0$ for each
$B \in \mathcal{X}_b$, endowed with the apparent
$\sigma$-field. Whenever $\psi$ is a measure on
$\mathbb{X} \times \mathbb{M}$ we write
$\bar\psi:=\psi(\cdot \times \mathbb{M})$ for its projection onto
$\mathbb{X}$. Conversely, if $\mu$ is a counting measure on $\BX$, we
construct a measure $\hat\mu$ on $\BX \times \mathbb{M}$ by endowing each
point of $\mu$ with a fixed (but arbitrary) mark $s \in\mathbb{M}$. 
Given some set $B \in \mathcal{X}$ and a measure
$\psi$ on $\mathbb{X} \times \mathbb{M}$, we write
$\psi_{B} := \psi\big( \cdot \cap (B \times \mathbb{M}) \big)$
for the restriction of $\psi$ onto $B \times \mathbb{M}$, and we denote in a
generic way by $\mathbf{N}_{fs}$ the set of finite simple counting measures. Here a counting measure is understood to be simple if it assigns to one-point-sets measure either $0$ or $1$.

For a point $(x, r) \in \mathbb{X} \times \mathbb{M}$ and a 
set $S \subset \mathbb{X} \times \mathbb{M}$ we write
$(x, r) \sim_\delta S$ if $(x, r) \sim_\delta (y, s)$ for some
$(y, s) \in S$, and likewise $(x, r) \not\sim_\delta S$ if $(x, r)$ is
not connected to any point in $S$ via $\sim_\delta$. We also use this
notation for $\psi\in \mathbf{N}(\mathbb{X} \times \mathbb{M})$,
formally meaning that $S$ is chosen as
\begin{align*}
	S = \supp\psi := \big\{(y,s) \in \mathbb{X} \times
	\mathbb{M} : \psi\big( \{ (y, s) \} \big) > 0 \big\} .
\end{align*}
We say that
$(x,r)$ and $(y,s)$ are connected via $\psi$ (and $\sim_\delta$),
written as $(x,r) \overset{\psi}{\sim}_\delta (y, s)$, if there
exist $k \in \mathbb{N}_0$ and
$(y_1, s_1), \dots, (y_k, s_k) \in \psi$ such that
$(y_j, s_j) \sim_\delta (y_{j + 1}, s_{j + 1})$ for $j = 0, \dots, k$,
with $(y_0, s_0) := (x, r)$ and $(y_{k + 1}, s_{k + 1}) := (y, s)$.

For $\delta > 0$, let
$\kappa_\delta\colon\mathbb{X} \times \mathbb{M} \times
\mathbf{N}(\mathbb{X} \times \mathbb{M}) \to [0, \infty)$ be given
through
\begin{align*}
\kappa_\delta(x, r, \psi)
:= \I\big\{ (x, r) \not\sim_\delta \psi \big\}
=\exp\bigg(-\int_{\mathbb{X} \times \mathbb{M}}- \log\big( \I\big\{ (x, r) \not\sim_\delta (y, s) \big\} \big) 
\, \mathrm{d}\psi(y, s) \bigg).
\end{align*}
The map $\kappa_\delta$ is obviously measurable and corresponds to
the PI of a pair interaction Gibbs process with hard core type pair
potential
$\big( (x, r), (y, s) \big) \mapsto \infty \cdot \I\big\{ (x, r)
\sim_\delta (y, s) \big\}$. For
$(x, r) \in \mathbb{X} \times \mathbb{M}$ and
$\psi \in \mathbf{N}(\mathbb{X} \times \mathbb{M})$ we call
\begin{align*}
C_\delta(x, r,\psi)
:= \int_{\mathbb{X} \times \mathbb{M}} \I\big\{ (y, s) \in \cdot \, \big\} \, \I\big\{ (x, r) 
\overset{\psi}{\sim}_\delta (y, s) \big\} \, \mathrm{d}\psi(y,s)
\end{align*}
the $\psi$-cluster of $(x,r)$ with respect to $\sim_\delta$. It is
easy to see that
$(x, r, \psi) \mapsto C_\delta(x, r, \psi) \in \mathbf{N}(\mathbb{X}
\times \mathbb{M})$ is a measurable mapping. Note that $(x, r) \not\sim_\delta \psi$
iff $C_\delta(x, r,\psi)=\mathbf{0}$,
where $\mathbf{0}$ denotes the null measure on $\BX \times \mathbb{M}$.
Therefore
\begin{align*}
\kappa_\delta(x, r, \psi)=\kappa_\delta\big( x, r, C_\delta(x, r, \psi) \big) .
\end{align*}

Define $\kappa\colon\mathbb{X} \times \mathbf{N}(\mathbb{X}) \to [0, \infty)$,
\begin{align*}
\kappa(x, \mu)
:= \exp\bigg( - \int_\mathbb{X} v(x, y) \, \mathrm{d}\mu(y) \bigg),
\end{align*}
the PI of a Gibbs process on $\mathbb{X}$ with pair potential $v$. As
in Section \ref{subGibbs}, we denote by $\mathrm{P}_{B,\nu}$ the
distribution of a (finite) Gibbs process with PI $\kappa^{(B,\nu)}$
(and reference measure $\lambda$), where $B \in \mathcal{X}_b$ and
$\nu \in \mathbf{N}(\mathbb{X})$. We now prove, in two steps, a
projection property of Gibbs processes. More specifically, we show
that (in the limit $\delta \downarrow 0$) the projection of a Gibbs
process on $\mathbb{X} \times \mathbb{M}$ with PI $\kappa_\delta$
onto $\mathbb{X}$ gives a Gibbs process with PI $\kappa$. The projection
property that is established in Proposition 2.1 of \cite{GeoHaeg96}, though
dealing specifically with the Potts model in $\R^d$, is conceptually
related.

\begin{lemma} \label{lemma 1} Let $\ell \in \mathbb{N}$ and
  $\psi\in \mathbf{N}_{fs}(\mathbb{X} \times \mathbb{M})$ with
  $\psi(B_\ell \times \mathbb{M}) = 0$.  For each $\delta > 0$, let
  $\xi_\delta$ be a Gibbs process on
  $\mathbb{X} \times \mathbb{M}$ with PI
  $\kappa_\delta^{(B \times \mathbb{M},\psi)}$ and reference
  measure $\lambda \otimes \mathbb{Q}$. 
Then, for every set
  $E \in \mathcal{N}(\mathbb{X})$,
\begin{align*}
\lim_{\delta \downarrow 0} \mathbb{P}\big(\bar\xi_\delta \in E \big)
=\mathrm{P}_{B_\ell, \bar\psi}(E).
\end{align*}
\end{lemma}
\begin{proof} For notational convenience we abbreviate $B:= B_\ell$.  For
  $\delta > 0$ and $E \in \mathcal{N}(\mathbb{X})$, the probability
  $\mathbb{P}\big( \bar\xi_\delta \in E \big)$ is given by
\begin{align*}
\frac{1}{Z_{\delta, B \times \mathbb{M}}(\psi)} \bigg[ \I_E(\mathbf{0})& 
+ \sum_{m = 1}^\infty \frac{1}{m!} \int_{B^m} \I_E\Big( \sum_{j=1}^m \delta_{x_j} \Big)  \\ 
		&\cdot \bigg( \int_{\mathbb{M}^m} 
\kappa_{\delta, m}(x_1, r_1, \dots, x_m, r_m,\psi) \, \mathrm{d}\mathbb{Q}^m(r_1, \dots, r_m) \bigg) 
\mathrm{d}\lambda^m(x_1, \dots, x_m) \bigg],
	\end{align*}
where $Z_{\delta, B \times \mathbb{M}}$ is the
partition function corresponding to the PI $\kappa_\delta$
and the measure $\lambda \otimes \mathbb{Q}$. Denote by
$y_1, \dots, y_k \in B^c$ the points of
$\bar\psi$. For $x_1, \dots, x_m \in B$ with
$x_1 \prec \dotso \prec x_m$ we have $x_m \prec y_j$ for each
$j \in \{1, \dots, k\}$, and we can find $\delta_0 > 0$ such
that the points $x_1, \dots, x_m, y_1, \dots, y_k$ lie in different sets of the $\delta$-partition for each
$\delta < \delta_0$. By definition of $\kappa_\delta$,
for each such choice we get
\begin{align*}
		\int_{\mathbb{M}^m}\kappa_{\delta, m}(x_1, r_1, \dots, x_m, r_m,\psi) \, 
\mathrm{d}\mathbb{Q}^m(r_1, \dots, r_m) 
&= \prod_{1 \leq i < j \leq m} \big( 1 - \varphi(x_i, x_j) \big) 
\, \prod_{i = 1}^m \prod_{j = 1}^k \big( 1 - \varphi(x_i, y_j) \big) \\
&= \kappa_m(x_1, \dots, x_m, \bar\psi) .
	\end{align*}
With the symmetry properties of $\kappa_{\delta, m}$ and the fact that $\lambda$ is diffuse, 
dominated convergence (using $\kappa_{\delta, m} \leq 1$) implies for each $F \in \mathcal{N}(\mathbb{X})$ that
\begin{align*}
  \lim_{\delta \downarrow 0} \sum_{m = 1}^\infty \frac{1}{m!} 
&\int_{B^m} \I_F\Big( \sum_{j = 1}^m \delta_{x_j} \Big) \cdot \bigg( \int_{\mathbb{M}^m}\kappa_{\delta, m}(x_1, r_1, \dots, x_m, r_m,\psi) 
\, \mathrm{d}\mathbb{Q}^m(r_1, \dots, r_m) \bigg) \mathrm{d}\lambda^m(x_1, \dots, x_m) \\
&= \sum_{m = 1}^\infty \frac{1}{m!} \int_{B^m} \I_F\Big( \sum_{j = 1}^m \delta_{x_j} \Big) 
\, \kappa_m(x_1, \dots, x_m, \bar\psi) \, \mathrm{d}\lambda^m(x_1, \dots, x_m) .
	\end{align*}
 Applied to
$F = \mathbf{N}(\mathbb{X})$ this yields
$\lim_{\delta \downarrow 0} Z_{\delta, B \times \mathbb{M}}(\psi) = Z_{B}(\bar\psi)$. 
A further application of the limit relation (to $F = E$), and the
observation from the beginning of this proof, imply the claim.
\end{proof}

Let $\mathrm{P}^\delta_{B_\ell \times \mathbb{M}, \psi}$ denote the
distribution of a (finite) Gibbs process with PI
$\kappa_\delta^{(B_\ell \times \mathbb{M}, \psi)}$ and
reference measure $\lambda \otimes \mathbb{Q}$. Then the previous lemma
reads as
\begin{align*}
  \lim_{\delta \downarrow 0}\mathrm{P}^\delta_{B_\ell \times \mathbb{M}, \psi}
\big(\big\{\nu\in\mathbf{N}(\mathbb{X}\times\mathbb{M}):\bar\nu \in E \big\} \big)
=\mathrm{P}_{B_\ell, \bar\psi}(E), \quad E \in \mathcal{N}(\mathbb{X}).
\end{align*}
However, this relation being true only for finite boundary conditions
$\psi$ is not enough to consider infinite range interactions. 
Fortunately the integrability assumption on the pair potential $v$ allows us to extract
more information. Recall that we associate with each $\mu\in \mathbf{N}(\mathbb{X})$ 
a measure $\hat\mu\in \mathbf{N}(\mathbb{X}\times\mathbb{M})$ by 
endowing each point of $\mu$ with a fixed  mark $s \in \mathbb{M}$. Also, keep in mind that any Gibbs process with a diffuse reference measure is simple, which follows immediately from the DLR equations \eqref{edlr} and the fact that a Poisson process with diffuse intensity measure is simple (cf. Proposition 6.9 of \cite{LastPenrose17}).

\begin{lemma} \label{lemma 2} Let $\ell \in \mathbb{N}$.
Let $\eta$ denote a Gibbs process on $\mathbb{X}$ with PI $\kappa$ and reference measure
  $\lambda$. Then we have for each $E \in \mathcal{N}(\mathbb{X})$ that
  \begin{align*}
\lim_{\delta \downarrow 0} ~ \mathbb{E} 
\Big|\mathrm{P}^\delta_{B_\ell \times \mathbb{M}, \hat\eta_{B_\ell^c}}
\big( \big\{\nu\in \mathbf{N}(\mathbb{X} \times \mathbb{M}):\bar\nu \in E \big\} 
\big) - \mathrm{P}_{B_\ell, \eta_{B_\ell^c}}(E) \Big|= 0.
\end{align*}
\end{lemma}
\begin{proof} As in the previous proof we abbreviate $B:=B_\ell$.
First of all, note that, for any $n > \ell$ and ($\lambda^m$-almost) all $x_1, \dots, x_m \in B$,
\begin{align} \label{eq lemma proof 1}
  \mathbb{E} \bigg| &\kappa_m(x_1, \dots, x_m, \eta_{B^c}) - 
\int_{\mathbb{M}^m}\kappa_{\delta, m}(x_1, r_1, \dots, x_m, r_m, \hat\eta_{B^c}) 
\, \mathrm{d}\mathbb{Q}^m(r_1, \dots, r_m) \bigg| \nonumber \\
  &\leq \mathbb{E} \big| \kappa_m(x_1, \dots, x_m, \eta_{B^c}) - 
\kappa_m(x_1, \dots, x_m, \eta_{B_n \setminus B}) \big| \nonumber \\
  &\quad + \mathbb{E} \bigg| \kappa_m(x_1, \dots, x_m, \eta_{B_n \setminus B}) - 
\int_{\mathbb{M}^m}\kappa_{\delta, m}(x_1, r_1, \dots, x_m, r_m, \hat\eta_{B_n \setminus B}) 
\, \mathrm{d}\mathbb{Q}^m(r_1, \dots, r_m) \bigg| \nonumber \\
  &\quad + \int_{\mathbb{M}^m} \mathbb{E} \big| 
\kappa_{\delta,m}(x_1, r_1, \dots, x_m, r_m, \hat\eta_{B_n \setminus B}) -
\kappa_{\delta, m}(x_1, r_1, \dots, x_m, r_m, \hat\eta_{B^c})
    \big| \, \mathrm{d}\mathbb{Q}^m(r_1, \dots, r_m).
\end{align}
By monotone convergence we have
$\lim_{n \to \infty} \kappa(x, \mu_{B_n})=\kappa(x,\mu)$
for all $x \in \mathbb{X}$ and
$\mu \in \mathbf{N}(\mathbb{X})$. Thus, by definition of
$\kappa_m$ and dominated convergence (using that $\kappa_m \leq 1$),
\begin{align*}
\lim_{n \to \infty} \, \mathbb{E} \big| \kappa_m(x_1, \dots, x_m, \eta_{B^c}) 
- \kappa_m(x_1, \dots, x_m, \eta_{B_n \setminus B}) \big|
= 0.
\end{align*}

For each fixed $n > \ell$, the proof of Lemma \ref{lemma 1} and dominated convergence yield
\begin{align*}
\lim_{\delta \downarrow 0} \, \mathbb{E} \bigg| \kappa_m(x_1, \dots, x_m, \eta_{B_n \setminus B}) 
- \int_{\mathbb{M}^m} \kappa_{\delta, m}(x_1, r_1, \dots, x_m, r_m, \hat\eta_{B_n \setminus B}) 
\, \mathrm{d}\mathbb{Q}^m(r_1, \dots, r_m) \bigg|
= 0.
\end{align*}
	
As for the last term in \eqref{eq lemma proof 1}, recall that
$\kappa_{\delta, m}$ is nothing but an indicator function. More precisely,
$\kappa_{\delta, m}(x_1, r_1, \dots, x_m, r_m,\hat\eta_{B^c})$ is equal to $1$ if there is no
$\sim_{\delta}$-connection between any of the points
$(x_1, r_1), \dots, (x_m, r_m)$ and none of these points is
connected to $\hat\eta_{B^c}$, and it is equal to $0$ otherwise. If
$\kappa_{\delta, m}(x_1, r_1, \dots, x_m, r_m,\hat\eta_{B^c})=1$ then clearly also
$\kappa_{\delta, m}(x_1, r_1, \dots, x_m, r_m,\hat\eta_{B_n \setminus B})=1$. Hence the only situation in
which the difference appearing in the last term of 
\eqref{eq lemma proof 1} can give a value different from $0$ is if
one of the points $(x_1, r_1), \dots, (x_m, r_m)$ is connected
to $\hat\eta_{B_n^c}$. Hence, we obtain for each $\delta>0$ and $n \in \N$ that
\begin{align*}
I_{\delta, n} &:=\int_{\mathbb{M}^m} 
\mathbb{E} \big|\kappa_{\delta, m}(x_1, r_1, \dots, x_m, r_m, \hat\eta_{B_n \setminus B}) 
-\kappa_{\delta, m}(x_1, r_1, \dots, x_m, r_m, \hat\eta_{B^c}) \big| 
\, \mathrm{d}\mathbb{Q}^m(r_1, \dots, r_m) \\
&\,\leq \int_{\mathbb{M}^m} \mathbb{E} \Big[ \I\big\{ (x_j, r_j) \sim_\delta \hat\eta_{B_n^c} 
\text{ for at least one } j \in \{ 1, \dots, m\} \big\} \Big] \, \mathrm{d}\mathbb{Q}^m(r_1, \dots, r_m) \\
&\,\leq \sum_{j = 1}^m \mathbb{E} \bigg[ \int_{B_n^c \times \mathbb{M}} 
\int_{\mathbb{M}^m} \I\big\{ (x_j, r_j) \sim_\delta (x, r) \big\} 
\, \mathrm{d}\mathbb{Q}^m(r_1, \dots, r_m) \, \mathrm{d} \hat\eta(x, r) \bigg].
\end{align*}
By construction, the marks of $\hat\eta$ 
are not used in any decision about connections in
the above term as $x_1, \dots, x_m$ are points in $B$
and thus always smaller than points in $\eta_{B_n^c}$ with
respect to the order $\prec$ on $\mathbb{X}$, so only the
marks $r_1, \dots, r_m$ matter. Thus, we arrive at
\begin{align*}
I_{\delta, n}&\leq\sum_{j = 1}^m \mathbb{E} \bigg[ \int_{B_n^c} \int_{\mathbb{M}^m} 
\I\big\{ (x_j, r_j) \sim_\delta (x, s) \big\} \, \mathrm{d}\mathbb{Q}^m(r_1, \dots, r_m) 
\, \mathrm{d} \eta(x) \bigg] \\
&= \sum_{j = 1}^m \mathbb{E} \bigg[ \int_{B_n^c} \varphi(x_j, x) \, \mathrm{d} \eta(x) \bigg] \\
&\leq \sum_{j = 1}^m \int_{B_n^c} \varphi(x_j, x) \, \mathrm{d} \lambda(x)
	\end{align*}
using that $\eta$ is a Gibbs process with PI $\kappa$ (and
$\kappa \leq 1$). This last term, however, goes to $0$ as
$n \to \infty$ by \eqref{e1.3}, and so does $I_{\delta, n}$ (uniformly in $\delta$). Therefore, the left-hand side of
\eqref{eq lemma proof 1} converges to $0$ as $\delta \downarrow 0$. 

Dominated convergence implies for each $E \in \mathcal{N}(\mathbb{X})$ that
\begin{align*}
\sum_{m = 1}^\infty \frac{1}{m!} &\int_{B^m} \I_E\Big( \sum_{j = 1}^m \delta_{x_j} \Big) \cdot \bigg( \int_{\mathbb{M}^m} 
\kappa_{\delta, m}(x_1, r_1, \dots, x_m, r_m, \hat\eta_{B^c}) 
\, \mathrm{d}\mathbb{Q}^m(r_1, \dots, r_m) \bigg) \mathrm{d}\lambda^m(x_1, \dots, x_m) \\
                          &\overset{L^1(\mathbb{P})}{\longrightarrow} ~
\sum_{m = 1}^\infty\frac{1}{m!} \int_{B^m} \I_E\Big( \sum_{j = 1}^m \delta_{x_j} \Big) 
\, \kappa_m(x_1,\dots,x_m,\eta_{B^c})\, \mathrm{d}\lambda^m(x_1,\dots,x_m)
	\end{align*}
as $\delta \downarrow 0$. It follows immediately that 
$Z_{\delta, B \times \mathbb{M}}(\hat\eta_{B^c}) 
\overset{L^1(\mathbb{P})}{\longrightarrow} Z_{B}(\eta_{B^c})$ 
as well as, for each $E \in \mathcal{N}(\mathbb{X})$,
\begin{align*}
Z_{\delta, B \times \mathbb{M}}(\hat\eta_{B^c}) \cdot 
\mathrm{P}^\delta_{B \times \mathbb{M}, \hat\eta_{B^c}}
\big( \big\{ \nu \in \mathbf{N}(\mathbb{X} \times \mathbb{M}) : 
\bar\nu \in E \big\} \big)
\overset{L^1(\mathbb{P})}{\longrightarrow} 
~ Z_{B}(\eta_{B^c}) \cdot \mathrm{P}_{B, \eta_{B^c}}(E)
\end{align*}
both as $\delta \downarrow 0$. Using that the occurring partition 
functions are always $\geq 1$ and bounded by $e^{\lambda(B)}$, the assertion 
follows from dominated convergence.
\end{proof}

Apart from this projection property of the hard core type Gibbs
processes in the extended state space, the construction via the
relation $\sim_{\delta}$ has another useful feature. It
allows an approximation of the RCM by considering a Poisson
process on $\BX \times \mathbb{M}$ and constructing the connections
via $\sim_{\delta}$. In the following result we show that in the limit
$\delta \downarrow 0$ the RCM is indeed recovered.

\begin{lemma} \label{lemma 3 approx of RCM} Let $\Psi$ be a
Poisson process on $\BX \times \mathbb{M}$ with intensity measure
$\lambda \otimes \mathbb{Q}$.
Let $\ell \in \N$,
$x \in B_\ell$ and
$\mu \in \mathbf{N}_{fs}(\BX) \cap \mathbf{N}_{B_\ell^c}(\BX)$. 
Denote by $\Gamma_{B_\ell}^{x, \mu}$
the RCM with connection function $\varphi = 1 - e^{-v}$ based
on $\bar\Psi_{B_\ell} + \mu + \delta_x$. Then, 
\begin{align*}
\lim_{\delta \downarrow 0} 
\int_{\mathbb{M}}\mathbb{P}\Big( (x, r) \overset{\Psi_{B_\ell}}{\sim_{\delta}} \hat\mu \Big)\,\mathrm{d}\mathbb{Q}(r)
= \mathbb{P}\Big( x \xleftrightarrow{\Gamma_{B_\ell}^{x, \mu}} \mu \Big).
\end{align*}
\end{lemma}
\begin{proof} Set $B:=B_\ell$.
Recalling that a Poisson process with diffuse intensity measure is simple, we have that, for $\mathbb{P}^{\bar\Psi_B}$-a.e.\ $\nu \in \mathbf{N}_{fs}(\BX) \cap \mathbf{N}_B(\BX)$ 
(write  $\nu = \sum_{j = 1}^k \delta_{x_j}$) the conditional distribution of
$\Psi_B$ given $\bar\Psi_B = \nu$ is
\begin{align*}
\int_{\mathbb{M}^k} \I\Big\{ \sum_{j = 1}^k \delta_{(x_j, r_j)} \in \cdot \, \Big\} \, 
\mathrm{d}\mathbb{Q}^k(r_1, \dots, r_k) .
\end{align*}
Choose $\delta > 0$ so small that the (finitely many) points in $\mu + \nu$ are
separated by the $\delta$-partition. 
By definition of $\sim_{\delta}$ (rendering the marks of $\hat\mu$ irrelevant)
and of the RCM we have
\begin{align*}
\int_{\mathbb{M}}\mathbb{P}
\Big( (x, r) \overset{\Psi_B}{\sim_{\delta}} \hat\mu \mid \bar\Psi_B = \nu \Big)\,\mathrm{d}\mathbb{Q}(r)
= \mathbb{P}\Big( x \xleftrightarrow{\Gamma_B^{x, \mu}} \mu \mid \bar\Psi_B = \nu \Big).
\end{align*}
Dominated convergence gives
\begin{align*}
\lim_{\delta \downarrow 0} 
\int_{\mathbb{M}}\mathbb{P}\Big( (x, r) \overset{\Psi_B}{\sim_{\delta}}\hat\mu \Big)\,\mathrm{d}\mathbb{Q}(r) 
&= \lim_{\delta \downarrow 0}\int_{\bN(\mathbb{X})} \int_\mathbb{M} 
\mathbb{P}\Big( (x, r) \overset{\Psi_B}{\sim_{\delta}} 
\hat\mu \mid \bar\Psi_B = \nu \Big) \, \mathrm{d}\mathbb{Q}(r) \, \mathrm{d}\mathbb{P}^{\bar\Psi_B}(\nu) \\
&= \int_{\bN(\mathbb{X})} \mathbb{P}\Big( x \xleftrightarrow{\Gamma_B^{x, \mu}} \mu \mid \bar\Psi_B = \nu \Big) 
\, \mathrm{d}\mathbb{P}^{\bar\Psi_B}(\nu) \\
&= \mathbb{P}\Big( x \xleftrightarrow{\Gamma_B^{x, \mu}} \mu \Big),
\end{align*}
as asserted.
\end{proof}

Just like for the projection property, a suitable approximation allows to
consider in Lemma \ref{lemma 3 approx of RCM} infinite boundary conditions (coming from a Gibbs
process). The proof of the 
following result shows why we have introduced the collection $\mathcal{X}_b^*$.

\begin{lemma} \label{lemma 4 approx of RCM unbounded} Let $\Psi$
be a Poisson process on $\BX \times \mathbb{M}$ with intensity
measure $\lambda \otimes \mathbb{Q}$. 
Let $\ell \in \N$ and $C \in \mathcal{X}$ with $C \subset B_\ell$. 
Suppose that $\eta$ is a Gibbs process on $\BX$ with PI $\kappa$ (i.e., with pair potential
$v$) and reference measure $\lambda$, such that $\eta$ is
independent of $\Psi$ and independent of the double sequence
that is used to construct the RCMs. 
For each $x\in\mathbb{X}$ we let $\Gamma_{\ell}^{x, \eta}$ be a RCM with
connection function $\varphi = 1 - e^{-v}$ based on
$\bar\Psi_{B_\ell} + \eta_{B_\ell^c} + \delta_x$. Then,
\begin{align*}
		\lim_{\delta \downarrow 0} \int_{C \times \mathbb{M}} 
\mathbb{P}\Big( (x, r) \overset{\Psi_{B_\ell}}{\sim_{\delta}} \hat\eta_{B_\ell^c} \Big) 
\, \mathrm{d}(\lambda \otimes \mathbb{Q})(x, r)
		= \int_C \mathbb{P}\Big( x \xleftrightarrow{\Gamma_{\ell}^{x, \eta}} \eta_{B_\ell^c} \Big) 
\, \mathrm{d}\lambda(x) .
\end{align*}
\end{lemma}
\begin{proof}
For $n \in \N$ with $n > \ell$ we write $\Gamma_{\ell, n}^{x, \eta}$
for the restriction of $\Gamma_{\ell}^{x, \eta}$ onto $B_n$, meaning
that only vertices inside $B_n$ and their connections among each
other remain. For such $n$, we have

\begin{align} \label{eq lemma proof 2}
  \bigg| \int_{C \times \mathbb{M}} 
&\mathbb{P}\Big( (x, r) \overset{\Psi_{B_\ell}}{\sim_{\delta}} \hat\eta_{B_\ell^c} \Big)
\,\mathrm{d}(\lambda \otimes \mathbb{Q})(x, r) - 
\int_C \mathbb{P}\Big( x \xleftrightarrow{\Gamma_{\ell}^{x, \eta}} \eta_{B_\ell^c} \Big) 
\, \mathrm{d}\lambda(x) \bigg| \nonumber \\
&\leq \int_{C \times \mathbb{M}} \Big| \mathbb{P}\Big( (x, r) \overset{\Psi_{B_\ell}}{\sim_{\delta}} 
\hat\eta_{B_\ell^c} \Big) - \mathbb{P}\Big( (x, r) \overset{\Psi_{B_\ell}}{\sim_{\delta}} 
\hat\eta_{B_n \setminus B_\ell} \Big) \Big| \, \mathrm{d}(\lambda \otimes \mathbb{Q})(x, r) \nonumber \\
&\quad+ \bigg| \int_{C \times \mathbb{M}} 
\mathbb{P}\Big((x,r)\overset{\Psi_{B_\ell}}{\sim_{\delta}} \hat\eta_{B_n \setminus B_\ell} \Big) 
\, \mathrm{d}(\lambda \otimes \mathbb{Q})(x, r) 
-\int_C \mathbb{P}\Big( x \xleftrightarrow{\Gamma_{\ell, n}^{x, \eta}} \eta_{B_n \setminus B_\ell} \Big) 
\, \mathrm{d}\lambda(x) \bigg| \nonumber \\
&\quad+ \int_C \Big|\mathbb{P}\Big(\xleftrightarrow{\Gamma_{\ell,n}^{x, \eta}} 
\eta_{B_n\setminus B_\ell} \Big) -\mathbb{P}\Big(x \xleftrightarrow{\Gamma_{\ell}^{x,\eta}} \eta_{B_\ell^c} \Big)
                                      \Big| \, \mathrm{d}\lambda(x) .
	\end{align}
	
We consider the three terms that appear in \eqref{eq lemma proof 2} separately. 
Let $\delta>0$.
As for the first term, note that
\begin{align*}
\int_{C \times \mathbb{M}} &\Big|\mathbb{P}\Big((x,r)\overset{\Psi_{B_\ell}}{\sim_{\delta}}\hat\eta_{B_\ell^c}\Big) 
- \mathbb{P}\Big( (x, r) \overset{\Psi_{B_\ell}}{\sim_{\delta}} \hat\eta_{B_n \setminus B_\ell} \Big) 
\Big| \, \mathrm{d}(\lambda \otimes \mathbb{Q})(x,r) \\
&\leq \int_{C \times \mathbb{M}} \mathbb{E} 
\Big| \I\Big\{ (x, r) \overset{\Psi_{B_\ell}}{\sim_{\delta}} \hat\eta_{B_\ell^c} \Big\} 
- \I\Big\{ (x, r) \overset{\Psi_{B_\ell}}{\sim_{\delta}} \hat\eta_{B_n \setminus B_\ell} \Big\} \Big| 
\, \mathrm{d}(\lambda \otimes \mathbb{Q})(x, r) \\
&\le\int_{C \times \mathbb{M}} 
\mathbb{E}\Big[\I\Big\{(x,r)\overset{\Psi_{B_\ell}}{\sim_{\delta}}\hat\eta_{B_n^c} \Big\} \Big] 
\, \mathrm{d}(\lambda \otimes \mathbb{Q})(x,r).
	\end{align*}
If $(x, r) \in C \times \mathbb{M}$ is $\sim_\delta$-connected via $\Psi_{B_\ell}$
to $\hat\eta_{B_n^c}$, then either $(x, r) \sim_\delta (y, s)$
for some point $y \in \eta_{B_n^c}$ or one of the Poisson
points is connected to $\hat\eta_{B_n^c}$. Thus, the previous
term is bounded by
\begin{align*}
\int_{C \times \mathbb{M}} &\mathbb{E}\bigg[ \int_{B_n^c} \I\big\{ (x, r) \sim_\delta (y, s) \big\} 
\, \mathrm{d}\eta(y) \bigg] \mathrm{d}(\lambda \otimes \mathbb{Q})(x, r) \\
&+ \lambda(C) \cdot \mathbb{E}\bigg[ \int_{B_\ell \times \mathbb{M}} 
\int_{B_n^c} \I\big\{ (z, t) \sim_\delta (y, s) \big\} \, \mathrm{d}\eta(y) \, \mathrm{d}\Psi(z, t) \bigg].
	\end{align*}
Using that $\eta$ is independent of $\Psi$ and that the
PI $\kappa$ is bounded by $1$, the above is bounded by
\begin{align*}
\int_{C \times \mathbb{M}} &\int_{B_n^c} \I\big\{ (x, r) \sim_\delta (y, s) \big\} \, \mathrm{d}\lambda(y) 
\, \mathrm{d}(\lambda \otimes \mathbb{Q})(x, r) \\
&+ \lambda(C) \int_{B_\ell \times \mathbb{M}} \int_{B_n^c} \I\big\{ (z, t) \sim_\delta (y, s) \big\} 
\, \mathrm{d}\lambda(y) \, \mathrm{d}(\lambda \otimes \mathbb{Q})(z,t).
	\end{align*}
As $B_\ell \subset B_n$, the points in $B_\ell$ are always
smaller (w.r.t. $\prec$) than the points in $B_n^c$, so by
construction of $\sim_\delta$ (and $\mathbb{Q}$), the latter
sum equals
\begin{align*}
\int_C \int_{B_n^c} \varphi(x, y) \, \mathrm{d}\lambda(y) \, \mathrm{d}\lambda(x) 
+ \lambda(C) \int_{B_\ell} \int_{B_n^c} \varphi(z, y) \, \mathrm{d}\lambda(y) \, \mathrm{d}\lambda(z) ,
\end{align*}
which converges to $0$ as $n \to \infty$ by dominated convergence,
using \eqref{eq. 3.1}. Hence, the first term on the right hand side of
\eqref{eq lemma proof 2} converges to $0$ as $n \to \infty$ uniformly
in $\delta$. 

The second term in \eqref{eq lemma proof 2} converges to
$0$ as $\delta \downarrow 0$ (for each fixed $n > \ell$) by the
independence of $\eta$ and $\Psi$, dominated convergence, and
Lemma \ref{lemma 3 approx of RCM}. 

As for the third term on the right
hand side of \eqref{eq lemma proof 2}, note that
\begin{align*}
\int_C \Big| \mathbb{P}\Big( x \xleftrightarrow{\Gamma_{\ell, n}^{x, \eta}} \eta_{B_n \setminus B_\ell}\Big) 
-\mathbb{P}\Big( x \xleftrightarrow{\Gamma_{\ell}^{x,\eta}} \eta_{B_\ell^c}\Big)\Big| \, \mathrm{d}\lambda(x)
\leq \int_C\mathbb{E}\Big|\I\Big\{x\xleftrightarrow{\Gamma_{\ell, n}^{x,\eta}}\eta_{B_n \setminus B_\ell} \Big\} 
-\I\Big\{ x \xleftrightarrow{\Gamma_{\ell}^{x, \eta}} \eta_{B_\ell^c} \Big\} \Big| \, \mathrm{d}\lambda(x).
\end{align*}
The difference of the indicator functions appearing in the
expectation can only be distinct from $0$ if there exists a
connection from $x \in C$ to $\eta_{B_n^c}$ (via $\Gamma_{\ell}^{x, \eta}$) which uses no point in
$\eta_{B_n \setminus B_\ell}$. Thus, either $x$ or one of the
Poisson points in $B_\ell$ has to be connected (directly) to
one of the points in $\eta_{B_n^c}$. Together with the given
independence properties, the quantity is therefore further
bounded by
\begin{align*}
\int_C \mathbb{E}\bigg[ &\int_{B_n^c} \varphi(x, y) \, \mathrm{d}\eta(y) \bigg]\mathrm{d}\lambda(x) 
+ \int_C \mathbb{E}\bigg[ \int_{B_\ell} \int_{B_n^c} \varphi(z, y) \, \mathrm{d}\eta(y) 
\,\mathrm{d}\bar\Psi(z) \bigg] \mathrm{d}\lambda(x) \\
&\le \int_C \int_{B_n^c} \varphi(x, y) \, \mathrm{d}\lambda(y) \, \mathrm{d}\lambda(x) 
+ \lambda(C) \int_{B_\ell} \int_{B_n^c} \varphi(z, y) \, \mathrm{d}\lambda(y) \, \mathrm{d}\lambda(z).
\end{align*}
By choice of $C$ and $B_\ell$, referring to \eqref{eq. 3.1},
dominated convergence implies that the above term converges to
$0$ as $n \to \infty$. Summarizing, we see that the left hand side
of \eqref{eq lemma proof 2} tends to zero as $\delta\downarrow 0$.
\end{proof}

Before we prove our main result, we need to investigate the
behavior of the RCM in the subcritical regime. More specifically, we
need to establish that the probability of a point $x$ being connected
to a (pair potential-) Gibbs process point in $B_\ell^c$ via the RCM
based on a Poisson process on $B_\ell$ goes to $0$ as
$\ell \to \infty$.

\begin{lemma} \label{lemma 5 behavior in subcrit. regime} Assume that
$(v, \lambda)$ is subcritical. Let $\Phi$ be a Poisson process on
$\BX$ with intensity measure $\lambda$. Suppose that $\eta$ is a Gibbs process with
PI $\kappa$ and reference measure $\lambda$, such that $\eta$ is
independent of $\Phi$ and independent of the double sequence that is
used to construct the RCMs. For each $x\in\mathbb{X}$ let $\Gamma_\ell^{x, \eta}$ 
be a RCM with connection function $\varphi = 1 - e^{-v}$ and vertex set
  $\Phi_{B_\ell} + \eta_{B_\ell^c} + \delta_x$. Then,
\begin{align*}
\lim_{\ell \to \infty} \mathbb{P}\Big( x \xleftrightarrow{\Gamma_\ell^{x, \eta}} \eta_{B_\ell^c} \Big)
=0,\quad \lambda\text{-a.e.\ $x\in\BX$}.
	\end{align*}
\end{lemma}
\begin{proof}
First of all, observe that since $\kappa \leq 1$, we can assume
without loss that $\eta \leq \Phi'$ almost surely, where $\Phi'$ is
a Poisson process with intensity measure $\lambda$ independent of
$\Phi$ and independent of the double sequence used for the RCM. This follows from an extension of Example 2.1 in \cite{GeorKun97} to unbounded $\lambda$, explicitly using that $\BX$ is a complete separable metric space (comparable to Lemma 5.3 in \cite{LastOtto21}). Thus, for each $x\in\mathbb{X}$ and each 
$\ell \in \N$, we have that
	\begin{align*}
		\mathbb{P}\Big( x \xleftrightarrow{\Gamma_\ell^{x, \eta}} \eta_{B_\ell^c} \Big)
		\leq \mathbb{P}\Big( x \xleftrightarrow{\Gamma_\ell^{x, \Phi'}} \Phi'_{B_\ell^c} \Big) ,
	\end{align*}
where $\Gamma_\ell^{x, \Phi'}$ denotes the RCM with connection function $\varphi$ based on
$\Phi_{B_\ell} + \Phi'_{B_\ell^c} + \delta_x$. Since the two Poisson processes are independent,
$\Phi_{B_\ell} + \Phi'_{B_\ell^c}$ is (for every $\ell \in \N$) a Poisson process on $\BX$ with intensity
measure $\lambda_{B_\ell} + \lambda_{B_\ell^c} = \lambda$. Therefore
(again using the independence), we can replace
$\Phi'_{B_\ell^c}$ by $\Phi_{B_\ell^c}$ and
$\Gamma_\ell^{x, \Phi'}$ by $\Gamma^x$ (the RCM based on
$\Phi + \delta_x$) in the above probability, which yields
\begin{align*}
\mathbb{P}\Big( x \xleftrightarrow{\Gamma_\ell^{x, \eta}} \eta_{B_\ell^c} \Big)
\leq \mathbb{P}\big( x \xleftrightarrow{\Gamma^{x}} \Phi_{B_\ell^c} \big) .
\end{align*} 
However, if $x$ is connected via $\Gamma^x$ to $\Phi_{B_\ell^c}$, then
the cluster of $x$ in $\Gamma^x$ has at least one point in
$B_\ell^c$. Thus, the probability in question is bounded by
\begin{align*}
		\mathbb{P}\big( C_x(B_\ell^c) > 0 \big) .
\end{align*}
Since $(v, \lambda)$ is assumed to be subcritical the latter probability tends to
$0$ as $\ell\to\infty$ for $\lambda$-a.e.\ $x\in\mathbb{X}$.
\end{proof}

We proceed to prove our main result. Though the technical details harnessed by the previous lemmata as well as our formal description differ from the proofs of existing uniqueness results, the last steps taken in the following proof retain a conceptual similarity to equation (4.4) of \cite{BergMaes94}, where the disagreement coupling was originally introduced for the discrete setting.

\begin{proof}[Proof of Theorem \ref{t2}]
Fix $\ell \in \mathbb{N}$ and let $\Psi$ be a Poisson
process on $\mathbb{X} \times \mathbb{M}$
with intensity measure $\lambda \otimes \mathbb{Q}$.
For each  $\psi,\psi' \in \mathbf{N}_{B_\ell^c \times \mathbb{M}}(\mathbb{X} \times \mathbb{M})$ and every
$\delta > 0$, Theorem 6.3 of \cite{LastOtto21} provides us
with a Gibbs process $\xi_\delta$ on $\mathbb{X} \times \mathbb{M}$ with PI
$\kappa_\delta^{(B_\ell \times \mathbb{M}, \psi)}$ and a Gibbs process $\xi_\delta'$ in
$\mathbb{X} \times \mathbb{M}$ with PI
$\kappa_\delta^{(B_\ell \times \mathbb{M}, \psi')}$
such that $\xi_\delta \leq \Psi$ and
$\xi_\delta' \leq \Psi$ almost surely, and such that
each point in $|\xi_\delta - \xi_\delta'|$ 
is $\sim_\delta$-connected via $\xi_\delta + \xi_\delta'$
to some point in $\psi+\psi'$. Hereby $|\nu|$ denotes the total variation measure of a signed measure $\nu$ on $\BX \times \mathbb{M}$, so, in the case of two counting measures $\nu, \nu' \in \mathbf{N}_{fs}(\BX \times \mathbb{M})$, the measure $|\nu - \nu'| \in \mathbf{N}_{fs}(\BX \times \mathbb{M})$ comprises those points in which $\nu$ and $\nu'$ differ.

For any $E \in \mathcal{N}_C(\mathbb{X})$ with $C \in \mathcal{X}$ and $C \subset B_\ell$, we obtain
\begin{align*}
  \big|\mathrm{P}^\delta_{B_\ell \times \mathbb{M},\psi}
  &\big( \big\{\nu \in \mathbf{N}(\mathbb{X} \times \mathbb{M}) : \bar\nu\in E \big\} \big) 
    - \mathrm{P}^\delta_{B_\ell \times \mathbb{M},\psi'}
\big( \big\{\nu \in \mathbf{N}(\mathbb{X} \times \mathbb{M}) : \bar\nu \in E \big\} \big) \big| \\
  &= \big| \mathbb{P}\big( \bar\xi_\delta \in E \big) - \mathbb{P}\big(\bar\xi'_\delta \in E \big) \big| \\
  &\leq \max\big\{ \mathbb{P}\big( \bar\xi_\delta \in E, \, \bar\xi'_\delta \notin E \big) , ~ 
\mathbb{P}\big(\bar\xi_\delta \notin E, \, \bar\xi'_\delta \in E \big) \big\} .
	\end{align*}
Since $E \in \mathcal{N}_C(\mathbb{X})$, the events in the
probability measures on the right hand side can only occur if
the restrictions of $\bar\xi_\delta$ and $\bar\xi'_\delta$ onto $C$
differ, so the term is bounded by
	\begin{align*}
\mathbb{P}\big( (\bar\xi_\delta)_C \neq (\bar\xi'_\delta)_C \big)
\leq \mathbb{P}\big( |\xi_\delta-\xi'_\delta|(C \times \mathbb{M}) > 0 \big) .
	\end{align*} 
Since each point in $|\xi_\delta - \xi_\delta'| \leq \Psi$ is
($\sim_\delta$-)connected via $\Psi_{B_\ell}$ to some
point in $\psi + \psi'$, a further bound is given through
\begin{align*}
\mathbb{E}\bigg[ \int_{C \times \mathbb{M}} &\I\Big\{ (x, r) \overset{\Psi_{B_\ell}}{\sim_{\delta}} 
(\psi + \psi') \Big\} \,\mathrm{d}\Psi(x, r) \bigg] \\
&= \mathbb{E}\bigg[ \int_{C \times \mathbb{M}} \I\Big\{ C_\delta(y,t,\Psi_{B_\ell})
\big( \{(x,r) \} \big) > 0 \text{ for some } (y, t) \in (\psi +\psi') \Big\} 
\,\mathrm{d}\Psi(x, r) \bigg] .
\end{align*}
By Mecke's equation, \cite[Theorem 4.1]{LastPenrose17}, this last term equals
\begin{align*}
\int_{C \times \mathbb{M}} &\mathbb{P} \Big( C_\delta\big(y, t,\Psi_{B_\ell} 
+ \delta_{(x, r)} \big)\big( \{ (x, r) \} \big) > 0 \text{ for some } (y,t) \in (\psi+\psi') \Big) 
\,\mathrm{d}(\lambda \otimes \mathbb{Q})(x, r) \\
&= \int_{C \times \mathbb{M}} \mathbb{P}\Big( (x, r) \overset{\Psi_{B_\ell}}{\sim_{\delta}} 
(\psi+\psi') \Big) \, \mathrm{d}(\lambda \otimes \mathbb{Q})(x, r).
\end{align*}

Now, let $\eta$ and $\eta'$ be two Gibbs processes on $\mathbb{X}$
with PI $\kappa$ and reference measure $\lambda$. Assume, without loss
of generality, that $(\eta, \eta')$ is independent of $\Psi$ and
independent of the double sequence used to define the RCMs. 
Let $C \in \mathcal{X}_b^*$ be arbitrary and choose $\ell$ large enough so that
$C \subset B_\ell$. Take $E \in \mathcal{N}_C(\mathbb{X})$. By the
DLR-equation \eqref{edlr}, Lemma \ref{lemma 2}, and the above bound, we obtain
\begin{align*}
\big| \mathbb{P}^{\eta}(E) - \mathbb{P}^{\eta'}(E) \big|
&= \Big| \mathbb{E}\Big[ \mathrm{P}_{B_\ell, \eta_{B^c_\ell}}(E) \Big] 
- \mathbb{E}\Big[ \mathrm{P}_{B_\ell, \eta'_{B_\ell^c}}(E) \Big] \Big| \\
		&\leq \mathbb{E} \Big| \mathrm{P}_{B_\ell, \eta_{B_\ell^c}}(E) - 
\mathrm{P}_{B_\ell, \eta'_{B_\ell^c}}(E) \Big| \\
&= \lim_{\delta \downarrow 0} \mathbb{E} 
\Big| \mathrm{P}^\delta_{B_\ell \times \mathbb{M}, \hat\eta_{B_\ell^c}}
\big(\big\{\nu \in \mathbf{N}(\mathbb{X} \times \mathbb{M}) : \bar\nu\in E \big\} \big) -\mathrm{P}^\delta_{B_\ell \times \mathbb{M}, \hat\eta'_{B_\ell^c}}
\big( \big\{\nu \in \mathbf{N}(\mathbb{X} \times \mathbb{M}) : \bar\nu\in E \big\} \big) \Big| \\
		&\leq \limsup_{\delta \downarrow 0} \int_{C \times \mathbb{M}} 
\mathbb{P}\Big( (x, r) \overset{\Psi_{B_\ell}}{\sim_{\delta}} (\hat\eta_{B_\ell^c} + \hat\eta'_{B_\ell^c}) \Big) 
\, \mathrm{d}(\lambda \otimes \mathbb{Q})(x, r) \\
&\leq \limsup_{\delta \downarrow 0} \int_{C \times \mathbb{M}} 
\mathbb{P}\Big( (x, r) \overset{\Psi_{B_\ell}}{\sim_{\delta}} \hat\eta_{B_\ell^c} \Big) 
\, \mathrm{d}(\lambda \otimes \mathbb{Q})(x, r) \\
		&\quad+ \limsup_{\delta \downarrow 0} \int_{C \times \mathbb{M}} 
\mathbb{P}\Big( (x, r) \overset{\Psi_{B_\ell}}{\sim_{\delta}} \hat\eta'_{B_\ell^c} \Big) 
\, \mathrm{d}(\lambda \otimes \mathbb{Q})(x, r) .
\end{align*}
Applying Lemma \ref{lemma 4 approx of RCM unbounded} and
dominated convergence to each of the two terms, we arrive at
\begin{align} \label{eq 1 proof theorem}
		\big| \mathbb{P}^{\eta}(E) - \mathbb{P}^{\eta'}(E) \big|
		\leq \int_{C} \mathbb{P}\Big( x \xleftrightarrow{\Gamma_{\ell}^{x, \eta}} \eta_{B_\ell^c} \Big) \, \mathrm{d}\lambda(x) + \int_{C} \mathbb{P}\Big( x \xleftrightarrow{\Gamma_{\ell}^{x, \eta'}} \eta'_{B_\ell^c} \Big) \, \mathrm{d}\lambda(x) .
	\end{align}
	Using Lemma \ref{lemma 5 behavior in subcrit. regime} and
        dominated convergences (for $\ell \to \infty$) twice, the
        right hand side of \eqref{eq 1 proof theorem} is seen to
        converge to $0$. As $C \in \mathcal{X}_b^*$ and
        $E \in \mathcal{N}_C$ were arbitrary, the measures
        $\mathbb{P}^{\eta}$ and $\mathbb{P}^{\eta'}$ agree on the
        algebra $\mathcal{Z}$ which generates $\mathcal{N}$. Hence,
        $\mathbb{P}^{\eta} \equiv \mathbb{P}^{\eta'}$ and the proof is
        complete.
\end{proof}

\section{Comments and examples}

In this section we first work in the general setting
of the introduction, that is we fix a complete separable metric space $(\BX,d)$
equipped with a locally finite measure $\lambda$, and let $v$ be a non-negative
pair potential. The next result is an immediate consequence of
Corollary \ref{c1}.

\begin{corollary}\label{c4.1} Assume that
\begin{align*}
\esssup_{x\in\BX}\int_{\mathbb{X}} \big(1-e^{-v(x,y)}\big)\,\mathrm{d}\lambda(y) < \infty.
\end{align*}
Then $|\mathcal{G}(v, \gamma\lambda)|=1$ for all sufficiently small $\gamma\ge 0$.
\end{corollary}

\begin{corollary}\label{c4.3} Assume that
\begin{align*}
\esssup_{x\in\BX}\int_{\mathbb{X}} v(x,y)\,\mathrm{d}\lambda(y) < \infty.
\end{align*}
Then $|\mathcal{G}(\beta v, \lambda)|=1$ for all sufficiently small $\beta\ge 0$.
\end{corollary}
\begin{proof} Since $1-e^{-\beta v}\le \beta v$ the result follows
from Corollary \ref{c1} and dominated convergence.
\end{proof}

The constant $\beta$ in Corollary \ref{c4.3} can be interpreted as {\em inverse temperature}.


\begin{example}\rm Suppose that $\BX$ equals the space $\mathcal{C}^d$
of all compact subsets of $\R^d$, equipped with the Hausdorff metric 
and a translation invariant locally finite measure $\lambda$, cf. \cite{SW08}. Let $V\colon \mathcal{C}^d\to[0,\infty]$ be measurable
with $V(\varnothing)=0$. For instance $V$ could be the volume or, if $\lambda$
is concentrated on the convex bodies, a linear combination
of the intrinsic volumes. Assume that the pair potential
is given by $v(K,L)=V(K\cap L)$, $K,L\in \mathcal{C}^d$.
As a percolation model,
the associated RCM with connection function $\varphi=1-e^{v}$
is considerably more general than the (Poisson driven) Boolean model,
studied (for spherical bodies), for instance, in \cite{MeesterRoy}.
The latter arises in the special case
$V(K)=\infty\cdot\I\{K\ne\varnothing\}$. Then the connection function
is given by $\varphi_\infty(K,L):=\I\{K\cap L\ne \varnothing\}$, so that
the connections do not involve any additional randomness. 
The corresponding Gibbs model are {\em hard particles in equilibrium},
while the case of a general $V$ could be addressed as
{\em soft particles in equilibrium}, at least if $V$ is translation invariant.

Theorem \ref{t2} requires $(\varphi,\lambda)$ to be subcritical,
while the previous results from \cite{HTHou17,BHLV20} (when specialized to
non-negative pair potentials) require the Boolean model
$(\varphi_\infty,\lambda)$ to be subcritical.
Since $\varphi(K,L)\leq \varphi_\infty(K,L)$ our result gives
better bounds on the uniqueness region. In particular,  
$\varphi(K,L)< \varphi_\infty(K,L)$ whenever $V(K\cap L)>0$.
If, for instance, $V$ is continuous at $\varnothing$,  
$\varphi(K,L)$ can be arbitrarily small, and still
$\varphi_\infty(K,L)=1$.
\end{example}

\begin{remark}\rm Since we can allow for a non-diffuse intensity
	measure, our results cover the case of a discrete graph $G=(V,E)$.
	We may then take $\BX=V$ and $\lambda=\gamma \lambda_0$, where
	$\lambda_0$ is the counting measure on $V$ and $\gamma>0$.
	A possible choice of a connection function is $\varphi(x,y):=p$
	if $\{x,y\}\in E$ and $\varphi(x,y)=0$, otherwise, where
	$p$ is a given probability. The resulting RCM
	is then a Poisson version of a {\em mixed percolation} model,
	see \cite{ChayesSch00}. However, $\varphi$ could also be {\em long-ranged}
	as in \cite{DeijfenHofstadHooghiemstra}. In fact, it easy to come up
	with a version of the model in \cite{DeijfenHofstadHooghiemstra}
	driven by a Poisson process on $\Z^d$. 
	In principle it might be possible to apply our uniqueness results
	to discrete models of statistical physics.
	We leave this for future research.
\end{remark}


\begin{remark}\rm It is believed (see e.g.\ \cite{Dereudre18}
and the references given there)
that in many Gibbs models there exist $\gamma^*>0$
such that $|\mathcal{G}(v,\gamma\lambda)|\ge 2$
provided that $\gamma>\gamma*$. We expect $\gamma_c$, as defined
by \eqref{criticalactivity}, to be (much) smaller than $\gamma^*$.
A careful analysis of our proofs suggests that a possible (but rather
implicit) approximation
of $\gamma^*$ is a critical intensity, which is defined in terms
of RCMs based on suitable finite volume versions of a Gibbs process
with  pair potential $v$, see e.g.\ \cite[Proposition 3.1]{GeoHaeg96}.
This is also supported by the discussion in \cite{GeorgiiLL05}.
\end{remark}

\section{Simulation results for the critical thresholds} \label{simulations}

Upon comparing the RCM with a branching process, our general Theorem \ref{t2} implies Corollary \ref{c1} which corresponds with the uniqueness result from \cite{HouZass21}. However, it is known from simulations that for the Gilbert graph, which corresponds to the hard sphere model, the branching bounds are widely off the actual critical intensity in lower dimensions. For an overview, we refer to \cite{Ziesche18}. Thus, we expect Theorem \ref{t2} to yield substantial improvements of Corollary \ref{c1} in low dimensions. To illustrate this point, we provide simulation results that give a rough overview on this difference in various models.

To approximate the true critical intensity of the RCM, we proceed as follows. For a given connection function in $\R^d$ with finite range $R > 0$, for instance one coming from a (finite range) pair potential, in a given dimension $d$ and for a given intensity $\gamma > 0$, we fix a large system size $S > R$. In the ball of radius $S$ around the origin we now construct the cluster of the origin in the RCM based on a stationary Poisson process with intensity $\gamma$ augmented by the origin. We start by simulating Poisson points (according to the given intensity) in the ball of radius $R$ around the origin, corresponding to all points which could possibly be connected to the origin and we check each of those points for such a connection (only to the origin). In the following, we keep track of three types of points, namely \textit{saturated} points which are part of the cluster and whose perspective has already been taken (which after the first step includes only the origin), those points which are part of the cluster but around which we might still have to simulate new Poisson points, and those Poisson points which are not yet connected to the cluster. Then we proceed algorithmically as follows. Of those cluster points from whose perspective we have not yet simulated we choose that point $x$ which is furthest away from the origin to take its perspective, meaning that we check if that particular point connects to any of the Poisson points which already exist but are not yet part of the cluster, and then proceed to simulate new Poisson points (according to the given intensity) in that part of $B(x, R)$ which was not yet covered in previous steps and we check if any of the new points connects to $x$. Note that we do not look for connections between two cluster points as we already know that both are part of the cluster and an additional connection between them does not change the size of the cluster. Also notice that when we first generate new points, we do not check for connections among them immediately (and only for connections to the center of the given step), but as soon as we take the perspective of any of the cluster points we check for connections to the points in its neighborhood and thus miss no relevant connection. The algorithm terminates as soon as all cluster points are saturated (and hence the construction of the cluster has died out within $B(0, S)$) or if the cluster connects to the complement of $B(0, S)$, meaning that some point in the cluster has a norm larger than $S$.
To make a decision whether the RCM percolates for a given intensity, we construct the cluster $5,000$ times. If the cluster connects to the complement of $B(0, S)$ a single time, we count that as percolation, even though, of course, a larger initial choice of the system size might have revealed that the cluster is actually finite. If the cluster lies within $B(0, S)$ in each of the $5,000$ runs, we count this as no percolation. To find a rough approximation of the critical intensity, we start at the branching lower bound, which is easily calculated for a given model, and increase the intensity by $10 \%$ as long as our algorithm decides that no percolation occurs. As soon as we first encounter percolation, we accept that intensity as an upper bound and the last intensity at which no percolation occurred as a lower bound. To refine the approximation further, we then slice the resulting interval by half two times, investigating the middle between the two bounds for percolation and adjusting the upper and lower bound accordingly. Note that the choices for $S$ and the number of runs where made according to our computational resources, larger values for both quantities will surely lead to better estimates, but one has to observe that there is no theoretical guarantee that the simulation will provide lower bounds for the critical intensities (see the discussion below). Note that a highly related algorithmic way for exploring a cluster is explained in Chapter 5.2 of \cite{JanLucRuc2000} for the discrete setting with finitely many vertices. 

To establish that the simulations yield plausible and useful results, we consider as a benchmark the hard sphere model with pair potential $v_R (x, y) = \infty \cdot \mathbf{1}\{ |x - y| \leq R \}$, $x, y \in \R^d$, which corresponds to the Gilbert graph with interaction range $R$, whose connection function is given through 
\begin{align*}
	\varphi_R(x, y) = 1 - \exp\big( - v(x, y) \big) = \mathbf{1}\big\{ |x - y| \leq R \big\} .
\end{align*}
The percolation properties of the Gilbert graph, in turn, are exactly those of the Boolean model with grains being balls with fixed radius $R/2$. The branching lower bound on the critical intensity is
\begin{align*}
	\bigg( \sup_{x \in \R^d} \int_{\R^d} \varphi_R(x, y) \, \mathrm{d}y \bigg)^{-1}
	= \frac{1}{V_d\big( B(0, R) \big)} .
\end{align*}
In the simulations, we fix $R = 2$ and compare the approximation of our method of simulation with the branching bounds and some of the best approximations for the critical intensity from the literature, namely the results from \cite{TorqJiao12}. Note that it is common in physics to investigate the percolation behavior in terms of the reduced number density, so we had to convert the values from \cite{TorqJiao12} into corresponding critical intensities by dividing with the volume of the unit ball in the given dimension. Throughout we round all values to five significant digits.

\begin{table*}[h] \label{Table 1}
	\centering
	\begin{tabular}{crclll} 
		\toprule
		{~~$d$~~} & \multicolumn{1}{c}{~$S$} & {~~Runs~} & \multicolumn{1}{c}{Approx. of $\gamma_c$} & \multicolumn{1}{c}{Branching bound} & \multicolumn{1}{c}{Torq./Jiao approximation \cite{TorqJiao12}} \\
		
		\cmidrule[0.4pt](r{0.125em}){1-1}
		\cmidrule[0.4pt](lr{0.125em}){2-2}
		\cmidrule[0.4pt](lr{0.125em}){3-3}
		\cmidrule[0.4pt](lr{0.125em}){4-4}
		\cmidrule[0.4pt](lr{0.125em}){5-5}
		\cmidrule[0.4pt](l{0.25em}){6-6}
		
		2  & ~1000 & ~5000 & \hspace{4mm} 0.34072 & \hspace{6mm} 0.079577 & \hspace{13mm} 0.35909 \\
		\myrowcolour
		3  & 500  & ~5000 & \hspace{4mm} 0.079338 & \hspace{6mm} 0.029842 & \hspace{13mm} 0.081621 \\
		4  & 300  & ~5000 & \hspace{4mm} 0.025915 & \hspace{6mm} 0.012665 & \hspace{13mm} 0.026435 \\
		\myrowcolour
		5  & 200  & ~5000 & \hspace{4mm} 0.010039 & \hspace{6mm} 0.0059368 & \hspace{13mm} 0.010342 \\
		\bottomrule
	\end{tabular}
	\caption{Gilbert graph}
\end{table*}

The comparison between our approximation and the lower bounds from \cite{TorqJiao12} (with the correction reported in \cite{TorqJiao14}), which are known to be very precise, show that our simulations, even with the very manageable choice of the simulation parameters, are reasonably well calibrated in that they provide conservative lower bounds for the critical intensity which are not wide off the mark and thus provide a solid reference for the order of magnitude of the critical intensity. The critical intensities of the Gilbert graph, the table immediately shows, are substantially larger than the branching lower bounds, which implies that our bounds on the region of uniqueness in the hard sphere model is substantially larger than the bounds by \cite{HouZass21}. For the hard spheres model this is not a new observation since our results agree (in this specific model) with the earlier disagreement percolation results by \cite{HTHou17}. This is due to the fact that, as discussed in Section \ref{sintro}, the range of the potential is the only information taken into account by \cite{HTHou17} and for the hard spheres model it happens to be the only relevant parameter. In the upcoming examples a further improvement can be observed.

Next we consider a modification of the hard sphere model, where an arbitrary overlap of spheres is possible, namely the penetrable spheres model considered by \cite{LikWatzLoew1998}. Let $0 < c < \infty$ and consider $v(x, y) = c \cdot \mathbf{1}\{ |x - y| \leq R \}$, $x, y \in \R^d$. The parameter $c$ (which in the hard sphere model is $\infty$) gives a measure on how valiantly spheres resist an overlap, but as $c$ is a fixed constant, the manner of the overlap plays no role in the spheres resistance of it. The RCM corresponding to this interaction function has connection function $\varphi(x, y) = (1 - e^{-c}) \cdot \mathbf{1}\{ |x - y| \leq R \}$. As we simulate from the RCM perspective, we parameterize $p = 1 - e^{-c} \in (0, 1)$ (in our case $p = 0.5$ and $p = 0.75$) which is then the probability that any two points with distance less than $R$ connect. Hence, the model can be interpreted as a modified Gilbert graph with an adjusted connection probability. The branching lower bounds for this RCM are simply those from the Gilbert graph divided by $p$. In order to be able to compare the simulation results for the different models, we again fix $R = 2$.

\begin{table*}[h] \label{Table 2}
	\centering
	\begin{tabular}{crcll} 
		\toprule
		{~~$d$~~} & \multicolumn{1}{c}{~$S$} & {~~Runs~} & \multicolumn{1}{c}{Approximation of $\gamma_c$} & \multicolumn{1}{c}{Branching bound} \\
		
		\cmidrule[0.4pt](r{0.125em}){1-1}
		\cmidrule[0.4pt](lr{0.125em}){2-2}
		\cmidrule[0.4pt](lr{0.125em}){3-3}
		\cmidrule[0.4pt](lr{0.125em}){4-4}
		\cmidrule[0.4pt](l{0.25em}){5-5}
		
		2  & ~1000 & ~5000 & \hspace{9mm} 0.48813 & \hspace{7mm} 0.15915 \\
		\myrowcolour
		3  & 500  & ~5000 & \hspace{9mm} 0.12503 & \hspace{7mm} 0.059683 \\
		4  & 300  & ~5000 & \hspace{9mm} 0.041814 & \hspace{7mm} 0.02533 \\
		\myrowcolour
		5  & 200  & ~5000 & \hspace{9mm} 0.01699 & \hspace{7mm} 0.011874 \\
		\bottomrule
	\end{tabular}
	\caption{Probability-adjusted Gilbert graph with $p = 0.5$}
\end{table*}

\begin{table*}[h] \label{Table 3}
	\centering
	\begin{tabular}{crcll} 
		\toprule
		{~~$d$~~} & \multicolumn{1}{c}{~$S$} & {~~Runs~} & \multicolumn{1}{c}{Approximation of $\gamma_c$} & \multicolumn{1}{c}{Branching bound}  \\
		
		\cmidrule[0.4pt](r{0.125em}){1-1}
		\cmidrule[0.4pt](lr{0.125em}){2-2}
		\cmidrule[0.4pt](lr{0.125em}){3-3}
		\cmidrule[0.4pt](lr{0.125em}){4-4}
		\cmidrule[0.4pt](l{0.25em}){5-5}
		
		2  & ~1000 & ~5000 & \hspace{9mm} 0.39376 & \hspace{6mm} 0.1061 \\
		\myrowcolour
		3  & 500  & ~5000 & \hspace{9mm} 0.096166 & \hspace{6mm} 0.039789 \\
		4  & 300  & ~5000 & \hspace{9mm} 0.03216 & \hspace{6mm} 0.016887 \\
		\myrowcolour
		5  & 200  & ~5000 & \hspace{9mm} 0.012459 & \hspace{6mm} 0.0079157 \\
		\bottomrule
	\end{tabular}
	\caption{Probability-adjusted Gilbert graph with $p = 0.75$}
\end{table*}

The observations in the penetrable spheres model are similar to those in the case of hard spheres. Even our conservative approximations of the critical intensity, and hence the region of uniqueness, improve the branching bounds (or Dobrushin method, \cite{HouZass21}) by factors larger than $3$ in two dimensions. In five dimensions the improvement is still by factors of more than $1.5$. Also the approximated values for $\gamma_c$ are substantially larger than in the Gilbert graph, which is an improvement over the classical disagreement percolation approach in \cite{HTHou17}, where only the range of $v$ is taken into account.

As a last model, we consider a pair interaction considered in the physics literature, namely the soft-sphere or inverse-power potential, which can be traced back at least to \cite{Row64}. Its interaction function is given through
\begin{align*}
	v(x, y) = \beta \cdot \frac{R^n}{|x - y|^n} \cdot \mathbf{1}\{ |x - y| \leq R \}, \quad x, y \in \R^d,
\end{align*}
where $\beta > 0$ is the characteristic energy and $n \in \N$ the hardness parameters. We fix $\beta = 1$ and consider $n \in \{ 6, 12 \}$. The connection function of the corresponding RCM is
\begin{align*}
	\varphi(x, y)
	= \bigg[ 1 - \exp\Big( - \frac{R^n}{|x - y|^n} \Big) \bigg] \cdot \mathbf{1}\big\{ |x - y| \leq R \big\} ,
\end{align*}
and the branching lower bounds for the critical intensity thus calculate as
\begin{align*}
	\Bigg( \frac{2 \pi^{d / 2}}{\Gamma\big( \tfrac{d}{2} \big)} \int_0^R \bigg[ 1 - \exp\Big( - \frac{R^n}{r^n} \Big) \bigg] \cdot r^{d - 1} \, \mathrm{d}r \Bigg)^{-1} .
\end{align*}
To ensure comparability, we consider $R = 2$.

\begin{table*}[h] \label{Table 4}
	\centering
	\begin{tabular}{crcll} 
		\toprule
		{~~$d$~~} & \multicolumn{1}{c}{~$S$} & {~~Runs~} & \multicolumn{1}{c}{Approximation of $\gamma_c$} & \multicolumn{1}{c}{Branching bound}  \\
		
		\cmidrule[0.4pt](r{0.125em}){1-1}
		\cmidrule[0.4pt](lr{0.125em}){2-2}
		\cmidrule[0.4pt](lr{0.125em}){3-3}
		\cmidrule[0.4pt](lr{0.125em}){4-4}
		\cmidrule[0.4pt](l{0.25em}){5-5}
		
		2  & ~1000 & ~5000 & \hspace{9mm} 0.35494 & \hspace{6mm} 0.084969 \\
		\myrowcolour
		3  & 500  & ~5000 & \hspace{9mm} 0.087095 & \hspace{6mm} 0.03276 \\
		4  & 300  & ~5000 & \hspace{9mm} 0.028471 & \hspace{6mm} 0.014254 \\
		\myrowcolour
		5  & 200  & ~5000 & \hspace{9mm} 0.01128 & \hspace{6mm} 0.0068329 \\
		\bottomrule
	\end{tabular}
	\caption{Soft-sphere model with $n = 6$}
\end{table*}

\begin{table*}[h] \label{Table 5}
	\centering
	\begin{tabular}{crcll} 
		\toprule
		{~~$d$~~} & \multicolumn{1}{c}{~$S$} & {~~Runs~} & \multicolumn{1}{c}{Approximation of $\gamma_c$} & \multicolumn{1}{c}{Branching bound}  \\
		
		\cmidrule[0.4pt](r{0.125em}){1-1}
		\cmidrule[0.4pt](lr{0.125em}){2-2}
		\cmidrule[0.4pt](lr{0.125em}){3-3}
		\cmidrule[0.4pt](lr{0.125em}){4-4}
		\cmidrule[0.4pt](l{0.25em}){5-5}
		
		2  & ~1000 & ~5000 & \hspace{9mm} 0.35272 & \hspace{6mm} 0.082379 \\
		\myrowcolour
		3  & 500  & ~5000 & \hspace{9mm} 0.083445 & \hspace{6mm} 0.031387 \\
		4  & 300  & ~5000 & \hspace{9mm} 0.027011 & \hspace{6mm} 0.013523 \\
		\myrowcolour
		5  & 200  & ~5000 & \hspace{9mm} 0.010873 & \hspace{6mm} 0.00643 \\
		\bottomrule
	\end{tabular}
	\caption{Soft-sphere model with $n = 12$}
\end{table*}

The simulations show that in the soft-sphere model, with the specific parameter specifications, Theorem \ref{t2} yields improvements on the region of uniqueness qualitatively similar to the penetrable spheres model.

The results of our simulations lead to the following observations. In the dimensions we consider, the true critical intensities are larger than the branching bounds by factors between $1.4$ and $4.5$, depending on the model and (mostly) on the dimension. Thus, Theorem \ref{t2} improves the regions of uniqueness of the corresponding Gibbs process by those very same factors as compared to \cite{HouZass21}. Moreover, for the penetrable and soft spheres, the critical values are (in some cases significantly) higher than those of the hard spheres which stands for an according improvement of the results by \cite{HTHou17}. While it is known from equation (6) of \cite{Pen96} that (for the Gilbert graph) the branching bound improves as the dimension grows, it also seems to improve (if much less so) if the overall connection probability in the RCM decreases and the critical intensity thus rises. 


Note that our approximation approach is really just an approximation and not founded on a solid theoretical basis. The benchmark model (hard spheres) and the corresponding existing simulations indicate that a further improvement of our approximations (by somewhere around $2$-$5\%$) is possible, but to our knowledge our simulations are the first for general RCMs. A different approach which would lead to approximations that come with confidence intervals is to prove a mean-field lower bound for the RCM in lines with equation (5.1) of \cite{Ziesche18} which holds for the Gilbert graph. This is an open problem that is certainly beyond the scope of this paper, but we strongly believe that a bound like this holds, at least for sufficiently regular connection functions. Still, our simulations for the Gilbert graph, where reference values from other simulations are available, suggest that our hands-on approach provides fairly solid and conservative approximations of the critical intensity of the RCM. Indeed, as we stay on the conservative side in every choice of simulation parameters, the results should slightly underestimate the corresponding true critical intensities but give a very good idea of their overall magnitude, in particular compared to the branching bounds.

\appendix

\section{Branching bounds on the random connection model}\label{sbranching}

In this section we provide a rigorous result as to when the RCM is subcritical. The proof is established via a bound on a suitable branching construction. We essentially work in the setting of Section \ref{subrcm}, that is, we consider a Borel space $(\BX, \mathcal{X})$ with a $\sigma$-finite measure $\lambda$ on $\BX$ and localizing structure $B_1 \subset B_2 \subset \cdots$ of sets with finite $\lambda$-measure. Let $\varphi : \BX \times \BX \to [0, 1]$ be a measurable and symmetric function and denote by $\Phi$ a Poisson process on $\BX$ with intensity measure $\lambda$. As before, we write $\Gamma$ (or $\Gamma^x$, or $\Gamma^{x, y}$) for the RCM with connection function $\varphi$ and vertex set $\Phi$ (or $\Phi + \delta_x$, or $\Phi + \delta_x + \delta_y$). We define the {\em pair connectedness} function $\tau : \BX \times \BX \to [0, 1]$,
\begin{align*}
\tau(x, y) := \mathbb{P}\big( x \xleftrightarrow{\Gamma^{x, y}} y \big) .
\end{align*}
We can immediately state the main result of this section, without introducing any further notation.
\begin{theorem} \label{theorem 2 branching bounds}
	Assume there exists a measurable function $g : \BX \to [0, \infty)$ such that
	\begin{align} \label{eq bound for branching}
	\int_\BX \varphi(x, y) \, \mathrm{d}\lambda(y) + \int_\BX \varphi(x, y) \, g(y) \, \mathrm{d}\lambda(y)
	\leq g(x) , \quad \lambda\text{-a.e. } x \in \BX .
	\end{align}
	Then, $(\varphi, \lambda)$ is subcritical.
\end{theorem}
Before we prove this result, we give some remarks.
\begin{remark}\rm
	It is easy to rewrite the theorem at hand in terms of $v$ if $\varphi$ is given through such a pair potential as in Section \ref{sproof}. A resemblance to \eqref{e1.3} then becomes obvious. In fact, we can strengthen assumption \eqref{e1.3} such that it implies the condition in the theorem. More precisely, assume that
	\begin{align*}
	q:= \esssup_{x \in \BX} \int_\BX \varphi(x, y) \, \mathrm{d}\lambda(y) < 1 .
	\end{align*}
	Then, choosing $g \equiv \frac{q}{1 - q}$, we obtain
	\begin{align*}
	\int_\BX \varphi(x, y) \, \mathrm{d}\lambda(y) + \int_\BX \varphi(x, y) \, g(y) \, \mathrm{d}\lambda(y)
	\leq q + \frac{q}{1 - q} \cdot q
	= \frac{q}{1 - q}
	= g(x) , \quad \lambda\text{-a.e. } x \in \BX .
	\end{align*}
	This is essentially the assumption in \cite{HouZass21}. The theorem in discussion is also weaker than assumption (KPU$_t$) in \cite{Jansen19}. Indeed, if there exists a measurable function $g : \BX \to [0, \infty)$ and some $t \geq 0$ such that
	\begin{align*}
	e^t \int_\BX \varphi(x, y) \, e^{g(y)} \, \mathrm{d}\lambda(y) \leq g(x) , \quad \lambda\text{-a.e. } x \in \BX,
	\end{align*}
	then \eqref{eq bound for branching} follows from $e^{g(y)} \geq 1 + g(y)$ and $e^t \geq 1$. In particular, Theorem \ref{theorem 2 branching bounds} improves Corollary C.1 in \cite{Jansen19}.
\end{remark}

\begin{proof}[Proof of Theorem \ref{theorem 2 branching bounds}.]
	First of all, notice that by an extension of Mecke's equation for the RCM, as stated in (4.1) of \cite{LastNestSchul21}, we have
	\begin{align*}
	\mathbb{E}\big[ C_x(\BX) \big] - 1
	= \mathbb{E}\bigg[ \int_\BX \I\big\{ x \xleftrightarrow{\Gamma^{x}} y \big\} \, \mathrm{d}\Phi(y) \bigg]
	= \int_\BX \mathbb{P}\big( x \xleftrightarrow{\Gamma^{x, y}} y \big) \, \mathrm{d}\lambda(y)
	= \int_\BX \tau(x, y) \, \mathrm{d}\lambda(y) .
	\end{align*}
	Observe that if two points $x, y \in \BX$ are connected via $\Gamma^{x, y}$, then either $x$ and $y$ are directly connected, or there lies at least one Poisson point in between them. Therefore,
	\begin{align*}
	\tau(x, y)
	\leq \varphi(x, y) + \mathbb{E}\bigg[ \int_\BX \varphi(x, z) \, \I\big\{ z \xleftrightarrow{\Gamma^{y}} y \big\} \, \mathrm{d}\Phi(z) \bigg] 
	= \varphi(x, y) + \int_\BX \varphi(x, z) \, \tau(z, y) \, \mathrm{d}\lambda(z) .
	\end{align*}
	Defining a convolution type operator in the obvious way, this inequality reads as
	\begin{align*}
	\tau \leq \varphi + \varphi * \tau .
	\end{align*}
	Iteration of this inequality yields that
	\begin{align*}
	\tau \leq \varphi + \varphi^{*2} + \dotso + \varphi^{*n} + (\varphi^{*n} * \tau)
	\end{align*}
	for every $n \in \N$. By a similar iteration of \eqref{eq bound for branching}, we see that
	\begin{align*}
	g(x) \geq \sum_{k = 1}^n \int_\BX \varphi^{*k}(x, y) \, \mathrm{d}\lambda(y) + \int_\BX \varphi^{*n}(x, y) \, g(y) \, \mathrm{d}\lambda(y) , \quad \lambda\text{-a.e. } x \in \BX ,
	\end{align*}
	for any $n \in \N$. In particular, we have
	\begin{align*}
	g(x) \geq \sum_{k = 1}^\infty \int_\BX \varphi^{*k}(x, y) \, \mathrm{d}\lambda(y) , \quad \lambda\text{-a.e. } x \in \BX .
	\end{align*}
	For each such $x \in \BX$ we thus have
	\begin{align*}
	\lim_{k \to \infty} \int_\BX \varphi^{*k}(x, y) \, \mathrm{d}\lambda(y) = 0 .
	\end{align*}
	Combining our previous observations, we find that, for $\lambda$-a.e. $x \in \BX$ and all $n, \ell \in \N$,
	\begin{align*}
	1 + \int_{B_\ell} \tau(x, y) \, \mathrm{d}\lambda(y)
	&\leq 1 + \sum_{k = 1}^n \int_{B_\ell} \varphi^{*k}(x, y) \, \mathrm{d}\lambda(y) + \int_{B_\ell} \int_\BX \varphi^{*n}(x, z) \, \tau(z, y) \, \mathrm{d}\lambda(z) \, \mathrm{d}\lambda(y) \\
	&\leq 1 + \sum_{k = 1}^n \int_{B_\ell} \varphi^{*k}(x, y) \, \mathrm{d}\lambda(y) + \lambda(B_\ell) \int_\BX \varphi^{*n}(x, z) \, \mathrm{d}\lambda(z) .
	\end{align*}
	Letting $n \to \infty$, we obtain
	\begin{align*}
	1 + \int_{B_\ell} \tau(x, y) \, \mathrm{d}\lambda(y)
	\leq 1 + \sum_{k = 1}^\infty \int_{B_\ell} \varphi^{*k}(x, y) \, \mathrm{d}\lambda(y)
	\end{align*}
	for $\lambda$-a.e. $x \in \BX$ and each $\ell \in \N$. With monotone convergence (let $\ell \to \infty$) we arrive at
	\begin{align*}
	\mathbb{E}\big[ C_x(\BX) \big]
	= 1 + \int_\BX \tau(x, y) \, \mathrm{d}\lambda(y)
	\leq 1 + \sum_{k = 1}^\infty \int_{\BX} \varphi^{*k}(x, y) \, \mathrm{d}\lambda(y)
	\leq 1 + g(x)
	< \infty
	\end{align*}
	for $\lambda$-a.e. $x \in \BX$. In particular, the cluster of each such $x$ (in $\Gamma^x$) is finite almost surely, that is to say, $(\varphi, \lambda)$ is subcritical.
\end{proof}

\section*{Acknowledgments}
We thank Sabine Jansen for several fruitful discussions of the topics surrounding our research.

\end{document}